\documentclass[reqno]{amsart}

\usepackage{amssymb,latexsym,mathrsfs,amsmath}
\usepackage{amsthm}
\usepackage{graphicx}
\usepackage{diagbox}
\usepackage{mathrsfs}
\usepackage{tikz}
\usetikzlibrary{matrix}
\usetikzlibrary{arrows}

\newtheorem{lemma}{Lemma}[section]
\newtheorem{proposition}[lemma]{Proposition}
\newtheorem{theorem}[lemma]{Theorem}
\newtheorem{corollary}[lemma]{Corollary}

\theoremstyle{definition}
\newtheorem{remark}[lemma]{Remark}
\newtheorem{definition}[lemma]{Definition}
\newtheorem{example}[lemma]{Example}

\newtheorem{question}[lemma]{Question}

\DeclareMathOperator{\id}{id}
\DeclareMathOperator{\chr}{char}

\DeclareMathOperator{\im}{im}

\DeclareMathOperator{\Hom}{Hom}
\DeclareMathOperator{\KR}{KR}

\newcommand{\Z}{\mathbb{Z}}

\newcommand{\Q}{\mathbb{Q}}
\newcommand{\R}{\mathbb{R}}
\newcommand{\C}{\mathbb{C}}
\newcommand{\F}{\mathbb{F}}
\newcommand{\K}{\mathbb{K}}

\newcommand{\Fr}{\mathcal{F}}
\newcommand{\FrMV}{\mathcal{F}_{\mathrm{MV}}}
\newcommand{\FrK}{\mathcal{F}_{\mathrm{Kh}}}
\newcommand{\Frof}{\mathcal{F}_{\bar{f}}}
\newcommand{\fa}{\mathbf{a}}
\newcommand{\fb}{\mathbf{b}}
\newcommand{\fc}{\mathbf{c}}
\newcommand{\of}{\bar{f}}
\newcommand{\ZMV}{\Z[\fa, \fb, \fc]}
\newcommand{\ZMVX}{\Z[\fa, \fb, \fc, X]}
\newcommand{\Fm}{\mathbf{Foam}}

\newcommand{\Csl}{C_{\mathfrak{sl}_3}}
\newcommand{\Hsl}{H_{\mathfrak{sl}_3}}

\newcommand{\pcrossing}[3] {
\draw[->, thick, >=latex] (#1, #2) -- (#1+#3, #2+#3);
\draw[-, thick, >=latex] (#1+#3, #2) -- (#1+ 0.6 * #3, #2 + 0.4 * #3);
\draw[->, thick, >=latex] (#1+ 0.4 * #3, #2 + 0.6 * #3) -- (#1, #2+#3);
}

\newcommand{\mcrossing}[3] {
\draw[->, thick, >=latex] (#1+#3, #2) -- (#1, #2+#3);
\draw[-, thick, >=latex] (#1, #2) -- (#1+ 0.4 * #3, #2 + 0.4 * #3);
\draw[->, thick, >=latex] (#1+ 0.6 * #3, #2 + 0.6 * #3) -- (#1+#3, #2+#3);
}

\newcommand{\jcrossing}[3] {
\draw[->, thick, >=latex] (#1+#3, #2) -- (#1, #2+#3);
\draw[->, thick, >=latex] (#1, #2) -- (#1+#3, #2+#3);
}

\newcommand{\smoothing}[4] {
\draw[#4, thick, >=latex] (#1, #2) to [out = 45, in = 315] (#1, #2+#3);
\draw[#4, thick, >=latex] (#1+#3, #2) to [out = 135, in = 225] (#1+#3, #2+#3);
}

\newcommand{\spidering}[4] {
\draw[#4, thick, >=latex] (#1, #2) -- (#1+0.5*#3, #2+0.25*#3);
\draw[#4, thick, >=latex] (#1+#3, #2) -- (#1+0.5*#3, #2+0.25*#3);
\draw[#4, thick, >=latex] (#1+0.5*#3, #2+0.75*#3) -- (#1, #2+#3);
\draw[#4, thick, >=latex] (#1+0.5*#3, #2+0.75*#3) -- (#1+#3, #2+#3);
\draw[#4, thick, >=latex] (#1+0.5*#3, #2+0.75*#3) -- (#1+0.5*#3, #2+0.25*#3);
}

\newcommand{\parallelbox}[4] {
\shade[color = gray!40, opacity = 0.3] (#1,#2) -- (#1 + #3, #2) -- (#1 + #3 + 0.3 * #4, #2 + #4) -- (#1 + 0.3 * #4, #2 + #4) -- (#1, #2);
\draw[-] (#1,#2) -- (#1 + #3, #2) -- (#1 + #3 + 0.3 * #4, #2 + #4) -- (#1 + 0.3 * #4, #2 + #4) -- (#1, #2);
}

\newcommand{\cylinder}[4]{
\shade[ball color = gray!40, opacity = 0.3] (#1,#2) arc (180:360:0.4*#3 and 0.2*#3) -- (#1+0.8*#3,#2+#4*#3) 
arc (0:180:0.4*#3 and -0.2*#3);
\shade[ball color = gray!40, opacity = 0.15] (#1,#2+#4*#3) arc (180:360:0.4*#3 and -0.2*#3) arc (0:180:0.4*#3 and -0.2*#3);
\draw (#1,#2) arc (180:360:0.4*#3 and 0.2*#3) -- (#1+0.8*#3,#2+#4*#3) 
arc (0:180:0.4*#3 and 0.2*#3) -- (#1,#2);
\draw[dashed] (#1+0.8*#3,#2) arc (0:180:0.4*#3 and 0.2*#3);
\draw (#1,#2+#4*#3) arc (180:360:0.4*#3 and 0.2*#3);
}

\newcommand{\cylinderd}[4]{
\shade[ball color = gray!40, opacity = 0.3] (#1,#2) arc (180:360:0.4*#3 and 0.2*#3) -- (#1+0.8*#3,#2+#4*#3) 
arc (0:180:0.4*#3 and -0.2*#3);
\shade[ball color = gray!40, opacity = 0.15] (#1,#2+#4*#3) arc (180:360:0.4*#3 and -0.2*#3) arc (0:180:0.4*#3 and -0.2*#3);
\draw (#1+0.8*#3, #2) -- (#1+0.8*#3,#2+#4*#3) arc (0:180:0.4*#3 and 0.2*#3) -- (#1,#2);
\draw (#1,#2+#4*#3) arc (180:360:0.4*#3 and 0.2*#3);
}

\newcommand{\cylindert}[4]{
\shade[ball color = gray!40, opacity = 0.3] (#1,#2) arc (180:360:0.4*#3 and 0.2*#3) -- (#1+0.8*#3,#2+#4*#3) 
arc (0:180:0.4*#3 and -0.2*#3);
\draw (#1, #2+#4 * #3) -- (#1,#2) arc (180:360:0.4*#3 and 0.2*#3) -- (#1+0.8*#3,#2+#4*#3) ;
\draw[dashed] (#1+0.8*#3,#2) arc (0:180:0.4*#3 and 0.2*#3);
}

\newcommand{\bowl}[4]{
\shade[ball color = gray!40, opacity = 0.3] (#1+0.8*#3,#2) arc (0:180:0.4*#3 and -0.2*#3) -- (#1,#2-#4*#3) arc (180:360:0.4*#3 and 0.4*#3) -- (#1+0.8*#3,#2);
\draw (#1+0.8*#3,#2) arc (0:180:0.4*#3 and -0.2*#3) -- (#1,#2-#4*#3) arc (180:360:0.4*#3 and 0.4*#3) -- (#1+0.8*#3,#2);
\draw (#1,#2) arc (180:360:0.4*#3 and -0.2*#3);
\shade[ball color = gray!40, opacity = 0.15] (#1,#2) arc (180:360:0.4*#3 and -0.2*#3) arc (0:180:0.4*#3 and -0.2*#3);
}

\newcommand{\bowludd}[4]{
\shade[ball color = gray!40, opacity = 0.3] (#1,#2) arc (180:360:0.4*#3 and 0.2*#3) -- (#1+0.8*#3,#2+#4*#3) arc (0:180:0.4*#3 and 0.4*#3) -- (#1,#2);
\draw (#1,#2) arc (180:360:0.4*#3 and 0.2*#3) -- (#1+0.8*#3,#2+#4*#3) arc (0:180:0.4*#3 and 0.4*#3) -- (#1,#2);
}

\newcommand{\bowlud}[4]{
\shade[ball color = gray!40, opacity = 0.3] (#1,#2) arc (180:360:0.4*#3 and 0.2*#3) -- (#1+0.8*#3,#2+#4*#3) arc (0:180:0.4*#3 and 0.4*#3) -- (#1,#2);
\draw (#1,#2) arc (180:360:0.4*#3 and 0.2*#3) -- (#1+0.8*#3,#2+#4*#3) arc (0:180:0.4*#3 and 0.4*#3) -- (#1,#2);
\draw[dashed] (#1+0.8*#3,#2) arc (0:180:0.4*#3 and 0.2*#3);
}

\newcommand{\sphere}[3]{
\shade[ball color = gray!40, opacity = 0.3] ({#1},{#2}) circle ({#3});
\draw (#1,#2) circle ({#3});
\draw (#1-#3,#2) arc (180:360:#3 and 0.3*#3);
\draw[dashed] (#1+#3,#2) arc (0:180:#3 and 0.3*#3);
}

\newcommand{\thetafoam}[3] {
\fill[color = gray, opacity = 0.7] (#1-#3,#2) arc (180:360:#3 and 0.3*#3) arc (0:180:#3 and 0.3*#3);
\shade[ball color = gray!40, opacity = 0.3] ({#1},{#2}) circle ({#3});
\draw (#1,#2) circle ({#3});
\draw[thick] (#1-#3,#2) arc (180:360:#3 and 0.3*#3);
\draw[dashed, thick] (#1+#3,#2) arc (0:180:#3 and 0.3*#3);
\draw[<-, >=latex, thick] (#1-0.2*#3,#2-0.3 * #3) -- (#1, #2-0.3 * #3);
}

\newcommand{\tripod}[5] {
\shade[color = gray!40, opacity = 0.2] (#1,#2) -- (#1-0.5, #2+0.5) -- (#1-0.5, #2+1.5) -- (#1,#2+1);
\draw (#1,#2) -- (#1-0.5, #2+0.5) -- (#1-0.5, #2+1.5) -- (#1,#2+1);
\shade[color = gray!40, opacity = 0.7] (#1,#2) -- (#1-0.4, #2-0.65) -- (#1-0.4, #2+0.35) -- (#1,#2+1);
\draw (#1,#2) -- (#1-0.4, #2-0.65) -- (#1-0.4, #2+0.35) -- (#1,#2+1);
\shade[color = gray!40, opacity = 0.3] (#1,#2) -- (#1+0.7, #2+0.1) -- (#1+0.7, #2+1.1) -- (#1,#2+1);
\draw[very thick] (#1,#2) -- (#1,#2+1);

\draw (#1,#2) -- (#1+0.7, #2+0.1) -- (#1+0.7, #2+1.1) -- (#1,#2+1);
\ifnum#3=1
\node at (#1-0.23, #2+0.1) {$\bullet$};
\fi
\ifnum#4=1
\node at (#1-0.27, #2+0.75) {$\bullet$};
\fi
\ifnum#5=1
\node at (#1+0.35, #2+0.55) {$\bullet$};
\fi
}

\newcommand{\bubble} {
\bowl{0}{0}{1}{0.2}
\fill[color = gray, opacity = 0.4] (-0.65,-0.325) -- (1, -0.325) -- (1.5, 0.325) -- (-0.15, 0.325); 
\draw (-0.65,-0.325) -- (1, -0.325) -- (1.5, 0.325) -- (-0.15, 0.325) -- (-0.65,-0.325); 
\fill[color = white] (0.4, 0) ellipse (0.4 and 0.21);
\draw[thick, color = gray] (0,0) arc (180:360:0.4 and -0.2) -- (0.8,0);
\bowludd{0}{0}{1}{0.2}
\draw[thick] (0,0) arc (180:360:0.4 and 0.2) -- (0.8,0);
\draw[<-, thick] (0.35, -0.2) -- (0.37, -0.2);
}

\newcommand{\cylinderslice} {
\cylindert{0}{-0.8}{1}{0.8}
\fill[color = gray, opacity = 0.4] (-0.65,-0.325) -- (1, -0.325) -- (1.5, 0.325) -- (-0.15, 0.325); 
\draw (-0.65,-0.325) -- (1, -0.325) -- (1.5, 0.325) -- (-0.15, 0.325) -- (-0.65,-0.325); 
\fill[color = white] (0.4, 0) ellipse (0.4 and 0.2);
\draw[thick, color = gray] (0,0) arc (180:360:0.4 and -0.2) -- (0.8,0);
\cylinderd{0}{0}{1}{0.8}
\draw[thick] (0,0) arc (180:360:0.4 and 0.2) -- (0.8,0);
\draw[<-, thick] (0.35, -0.2) -- (0.37, -0.2);
}

\newcommand{\resolveslice} {
\bowlud{0}{-0.8}{1}{0}
\fill[color = gray, opacity = 0.4] (-0.65,-0.325) -- (1, -0.325) -- (1.5, 0.325) -- (-0.15, 0.325); 
\draw (-0.65,-0.325) -- (1, -0.325) -- (1.5, 0.325) -- (-0.15, 0.325) -- (-0.65,-0.325); 
\bowl{0}{0.8}{1}{0}
}

\newcommand{\resolvetop} {
\cylindert{0}{-1}{1}{1}
\fill[color = gray, opacity = 0.4] (-0.65,-0.325) -- (1, -0.325) -- (1.5, 0.325) -- (-0.15, 0.325); 
\draw (-0.65,-0.325) -- (1, -0.325) -- (1.5, 0.325) -- (-0.15, 0.325) -- (-0.65,-0.325); 
\fill[color = white] (0.4, 0) ellipse (0.4 and 0.2);
\draw[thick, color = gray] (0,0) arc (180:360:0.4 and -0.2) -- (0.8,0);
\bowludd{0}{0}{1}{0}
\bowl{0}{1}{1}{0}
\draw[thick] (0,0) arc (180:360:0.4 and 0.2) -- (0.8,0);
\draw[<-, thick] (0.35, -0.2) -- (0.37, -0.2);
}

\newcommand{\resolvebot} {
\bowlud{0}{-1}{1}{0}
\bowl{0}{0}{1}{0.1}
\fill[color = gray, opacity = 0.4] (-0.65,-0.325) -- (1, -0.325) -- (1.5, 0.325) -- (-0.15, 0.325); 
\draw (-0.65,-0.325) -- (1, -0.325) -- (1.5, 0.325) -- (-0.15, 0.325) -- (-0.65,-0.325); 
\fill[color = white] (0.4, 0) ellipse (0.4 and 0.2);
\draw[thick, color = gray] (0,0) arc (180:360:0.4 and -0.2) -- (0.8,0);
\bowludd{0}{0}{1}{0}
\bowl{0}{1}{1}{0}
\draw[thick] (0,0) arc (180:360:0.4 and 0.2) -- (0.8,0);
\draw[<-, thick] (0.35, -0.2) -- (0.37, -0.2);
}

\begin{document}
\parindent0em
\setlength\parskip{.1cm}
\thispagestyle{empty}
\title{On Variations of $S$-invariants from $\mathfrak{sl}_3$-link homology}
\author[Dirk Sch\"utz]{Dirk Sch\"utz}
\address{Department of Mathematical Sciences\\ Durham University\\ Durham DH1 3LE\\ United Kingdom}
\email{dirk.schuetz@durham.ac.uk}
%\subjclass[2020]{primary: 57K18; secondary: 57K10}
%\keywords{Odd Khovanov homology, Steenrod square}

\begin {abstract}
We use the Mackaay--Vaz universal $\mathfrak{sl}_3$-link homology to deepen the study of $s$-invariants on Khovanov's link homology associated to $\mathfrak{sl}_3$. Such $s$-invariants have already been studied by Lobb and Wu in characteristic $0$ and we show how to extend this to other characteristics, particularly to $p=3$. We also use Bar-Natan's scanning algorithm for efficient calculations of these invariants, and exhibit more examples of unusual behaviour that has been previously observed by Lewark--Lobb.
\end{abstract}

\maketitle

\section{Introduction}
Rasmussen's groundbreaking work \cite{MR2729272} on a lower bound for the slice genus of a knot coming from Khovanov homology \cite{MR1740682} has been generalized in many different directions. One of these directions was pursued independently by Lobb \cite{MR2554935} and Wu \cite{MR2509322} using Khovanov--Rozansky $\mathfrak{sl}_n$-link homology \cite{MR2391017} for $n\geq 3$.

Many of the formal properties of Rasmussen's $s$-invariant carry over to the setting using $\mathfrak{sl}_n$, $n\geq 3$, but one big difference lies in computability. While the $s$-invariant can easily be computed in bulk for knots with up to $20$ crossings, and even knots with $60$ crossings do not represent a serious difficulty, the same can certainly not be said for $\mathfrak{sl}_n$-link homology. However, for $n=3$ Lewark \cite{MR3248745} showed that the scanning technique of Bar-Natan \cite{MR2320156} can be adapted for fast calculations of $\mathfrak{sl}_3$-link homology as defined in \cite{MR2100691}, and for general $n$ this can be done provided the knot is {\em bipartite}, see Lewark--Lobb \cite{MR3458146}.

Whether a knot is bipartite or not is generally difficult to tell, but certain pretzel knots are, compare \cite{MR3458146}, and Lewark--Lobb were able to use these knots to show that the resulting $s$-invariants for $\mathfrak{sl}_n$ can be quite different for various $n$, and that other interesting phenomena appear for $n\geq 3$. One of these phenomena is that there are more spectral sequences starting from $\mathfrak{sl}_n$-link homology than one might expect, and that their $E_\infty$-pages can even lead to different slice obstructions.

In this paper, we focus on the case of $\mathfrak{sl}_3$-link homology. This allows us to perform reasonably fast computations without having to restrict ourselves to bipartite knots. Furthermore, Lewark \cite{MR3248745} restricted himself to computations of the standard $\mathfrak{sl}_3$-link homology, making $s$-invariant computations somewhat reliant on being able to infer the spectral sequence from the homology.  

The calculations of $s$-invariants in \cite{MR3248745, MR3458146} are done in characteristic $0$, but recent results for standard Khovanov homology, compare \cite{LewarkZib, MR4873797, dunfield2024, lewark2024new}, suggest that different characteristics should also be of interest for $\mathfrak{sl}_3$-link homology. The extension to characteristics $p\not=3$ is straightforward, but we also get a homomorphism $s^3_{\mathfrak{sl}_3}\colon \mathfrak{C}\to \Z$ in this case, where $\mathfrak{C}$ is the smooth concordance group of knots.

\begin{theorem}\label{thm:first_main}
The homomorphism $s^3_{\mathfrak{sl}_3}\colon \mathfrak{C}\to \Z$ satisfies
\begin{enumerate}
\item $g_4(K) \geq |s^3_{\mathfrak{sl}_3}(K)|/2$ for each knot $K$, where $g_4$ refers to the $4$-genus of $K$.
\item $s^3_{\mathfrak{sl}_3}(T(p,q)) = (p-1)(q-1)$ for the $(p,q)$-torus knot $T(p,q)$.
\item The image of $s^3_{\mathfrak{sl}_3}$ is generated by the image of $T(2,3)$.
\end{enumerate}
\end{theorem}

The first two conditions mean that $s^3_{\mathfrak{sl}_3}$ is (up to a scalar) a {\em slice-torus invariant}, compare \cite{MR2057779}. The third condition is worth mentioning, since for $p\not=3$ the homomorphism only satisfies (1) and (2), but not (3). This puts $s^3_{\mathfrak{sl}_3}$ closer to the standard Rasmussen $s$-invariants, but we will see that they are not the same. It also shows that $s^3_{\mathfrak{sl}_3}$ is linearly independent of $s^0_{\mathfrak{sl}_3}$ and we will also see that $p = 0, 2, 3$ give linearly independent homomorphisms.

The aforementioned spectral sequences, which can also be used to calculate $s^p_{\mathfrak{sl}_3}$, are closely related to Frobenius algebras of the form $\F[X]/f(X)$, where $f(X)$ is a separable polynomial of degree $3$. In particular, $X^3-1$ is used for $p\not=3$ in the definition of $s^p_{\mathfrak{sl}_3}$, and $X^3-X$ for $p=3$. 

Concordance invariants can also be derived from other separable polynomials, but as was observed in \cite{MR3458146} need not give rise to homomorphisms from the concordance group, at least not in a straightforward way. Lobb--Lewark \cite[\S 3]{MR3458146} introduced a notion of $\KR$-equivalence (depending on a knot $K$) on polynomials over $\C$ to get a better grasp on the various invariants. While $X^3-1$ and $X^3-X$ are easily seen to not be $\KR$-equivalent, finding other separable polynomials not $\KR$-equivalent to $X^3-X$ is more difficult. Lobb--Lewark showed that $X^3-X-1$ is not $\KR$-equivalent to $X^3-X$ for certain connected sums of pretzel knots.

We make calculations that show that for some torus knots with $4$, $5$, or $6$ strands the polynomials $X^3-X$ and $X^3-X-1$ are not $\KR$-equivalent. We also give a criterion (see Proposition \ref{prp:kr_crit}) that explains why $X^3-X$ and $X^3-X-1$ are so often $\KR$-equivalent for knots with a small number of crossings.

\subsection*{Acknowledgements} The author would like to thank Lukas Lewark for useful comments and clearing up a few technicalities, and Andrew Lobb for useful discussions.

\section{Rank Three Frobenius systems}

A {\em Frobenius system} $\Fr=(\Lambda,A,\varepsilon, \Delta)$ consists of an inclusion of commutative rings $\imath\colon \Lambda\to A$, a $\Lambda$-module map $\varepsilon\colon A\to \Lambda$, and a $A$-bimodule map $\Delta\colon A\to A\otimes_\Lambda A$ which is co-associative and co-commutative, such that $(\varepsilon\otimes \id)\Delta = \id$. We say that $\Fr$ is of rank $3$, if there exists $X\in A$ such that $1, X, X^2$ is a basis of $A$ as a $\Lambda$-module.

We note that all our rings are assumed to have an identity $1$, and any ring homomorphism sends $1$ to $1$.

\begin{example}
The {\em Mackaay--Vaz system} 
\[
\FrMV = (\ZMV, \ZMVX/(X^3-\fa X^2-\fb X-\fc), \varepsilon, \Delta)
\]
is given by
\[
\varepsilon(1) = 0,\hspace{1cm}\varepsilon(X) = 0, \hspace{1cm}\varepsilon(X^2) = -1,
\]
and
\begin{align}\label{eq:MV_comult}
\nonumber&\Delta(1) = -1\otimes X^2 - X\otimes X - X^2\otimes 1 + \fa (1\otimes X + X \otimes 1) + \fb \otimes 1,\\
&\Delta(X) = -X\otimes X^2 - X^2\otimes X +\fa X\otimes X - \fc \otimes 1,\\
\nonumber&\Delta(X^2) = -X^2\otimes X^2 - \fb X\otimes X - \fc (1\otimes X + X\otimes 1).
\end{align}
This Frobenius system has a $q$-grading given by
\[
|1|_q = 0, \hspace{0.5cm} |\fa|_q = 2, \hspace{0.5cm} |\fb|_q = 4, \hspace{0.5cm}|\fc|_q = 6, \hspace{0.5cm}|X|_q = 2.
\]
\end{example} 

Following Khovanov \cite{MR2232858}, a ring homomorphism $\psi\colon \Lambda \to \Lambda'$ induces a Frobenius system $(\Lambda', A', \varepsilon',\Delta')$, called a {\em base change}, where $A' = A\otimes_\Lambda \Lambda'$, $\varepsilon' = \varepsilon\otimes \id_{\Lambda'}$, and $\Delta' = (\imath'\otimes \id_{A'})\circ (\Delta\otimes \id_{\Lambda'})$, where $\imath'\colon A\to A'$ is given by $\imath'(a) = a\otimes 1$.

\begin{example}\label{ex:easy_change}
Let $\Lambda$ be a commutative ring and $a,b,c\in \Lambda$. Define the polynomial 
\begin{equation}\label{eq:poly_standard}
f(X) = X^3 - aX^2-bX-c.
\end{equation}
Then $\Fr_f=(\Lambda, \Lambda[X]/(f(X)), \varepsilon, \Delta)$ is the base change from $\FrMV$ via the ring homomorphism $\psi\colon \Z[\fa,\fb,\fc]\to \Lambda$ sending $1$ to $1$, $\fa$ to $a$, $\fb$ to $b$, and $\fc$ to $c$. If $a=b=c =0$ we get the Frobenius system $\FrK$ considered by Khovanov in \cite{MR2100691}. If $\Lambda=\C$, we get the Frobenius systems from Mackaay--Vaz in \cite[\S 3]{MR2336253}.
\end{example}

We want to keep the grading of $\FrMV$ intact, while also having a more flexible system as in \cite[\S 3]{MR2336253}. We will therefore be mainly interested in the following systems.

\begin{example}
Let $\K$ be a commutative ring, $a,b,c\in \K$ and $\Lambda = \K[h]$, and consider the polynomial
\begin{equation}\label{eq:poly_graded}
\of(X) = X^3 - ahX^2 -bh^2X-ch^3.
\end{equation}
Define $A = \K[h, X]/(\of(X))$. We then get a Frobenius system $\Frof = (\K[h], A, \varepsilon, \Delta)$ as a base change from $\FrMV$. Setting $|h|_q = 2$ turns it into a graded Frobenius system (with elements of $\K$ having $q$-grading $0$). The base change $\K[h] \to \K$ sending $h$ to $1$ recovers Example \ref{ex:easy_change} for general $a,b,c\in \K$, while sending $h$ to $0$ recovers $\FrK$.
\end{example}

Any base change from $\FrMV$ is a system $\Fr_f=(\Lambda, A,\varepsilon, \Delta)$, with $A = \Lambda[X]/(f(X))$ for some $f(X)$ as in (\ref{eq:poly_standard}). Given $\alpha\in \Lambda$, define
\[
f'(X) = X^3-(a-3\alpha)X^2-(b+2a\alpha-3\alpha^2)X-(c+b\alpha+a\alpha^2-\alpha^3),
\]
and $A' = \Lambda[X]/(f'(X))$.

\begin{lemma}\label{lm:easyiso}
The Frobenius systems $\Fr_f$ and $\Fr_{f'}$ are isomorphic via the ring isomorphism $\Phi_\alpha\colon A\to A'$ given by $\Phi_{\alpha}(X) = X+\alpha$.
\end{lemma}

\begin{proof}
The polynomial $f'(X)$ is chosen exactly so that the ring homomorphism $\Phi\colon \Lambda[X]\to\Lambda[X]$ mapping $X$ to $X+\alpha$ sends $f(X) = X^3-aX^2-bX-c$ to $f'(X)$. Hence $\Phi_\alpha$ is a well defined ring homomorphism. It is also an isomorphism with inverse $X\mapsto X-\alpha$. It remains to show that $\Phi_\alpha$ commutes with co-units and co-multiplication. For the co-unit we have
\[
\varepsilon'(\Phi_\alpha(1)) = \varepsilon'(\Phi_\alpha(X)) = 0, \hspace{1cm}\varepsilon'(\Phi_\alpha(X^2)) = -1,
\]
so $\varepsilon'\circ \Phi_\alpha = \varepsilon$. For the co-multiplication, we only need to check that
\[
\Phi_\alpha\otimes \Phi_\alpha \circ \Delta(1) = \Delta'\circ\Phi_\alpha(1),
\] 
since $\Delta$ and $\Delta'$ are bimodule maps. A straightforward calculation reveals that this is indeed the case.
\end{proof}

The {\em dual system} of the Frobenius system $\Fr = (\Lambda, A, \varepsilon, \Delta)$ is the Frobenius system $\Fr^\ast=(\Lambda, A^\ast, \imath^\ast, m^\ast)$, where $A = \Hom_\Lambda(A,\Lambda)$, $m^\ast$ is the dual map of the multiplication map on $A$, and $\imath^\ast$ is the dual of the inclusion $\Lambda\subset A$. If $\Fr$ is isomorphic to $\Fr^\ast$ as Frobenius systems, we call $\Fr$ {\em self-dual}.

\begin{proposition}\label{prp:self_dual}
The Frobenius system $\FrMV$ is self-dual.
\end{proposition}

\begin{proof}
Since $\FrMV$ is a rank $3$ system, we have that $A^\ast$ is a free $\ZMV$-module of rank $3$. Define $X^\ast \in A^\ast$ by
\[
X^\ast(1) = 0,\hspace{1cm}X^\ast(X) = -1, \hspace{1cm}X^\ast(X^2) = 0.
\]
It follows from (\ref{eq:MV_comult}) that
\[
X^\ast\otimes X^\ast\circ \Delta(1) = -1, \hspace{0.5cm} X^\ast\otimes X^\ast\circ\Delta(X) = \fa, \hspace{0.5cm} X^\ast\otimes X^\ast\circ\Delta(X^2) = -\fb.
\]
Therefore $\varepsilon, X^\ast, X^{\ast\,2}$ is a basis of $A^\ast$, and $\FrMV^\ast$ is a rank $3$ system. Now consider $X^{\ast\,3}$. We have
\begin{align*}
X^{\ast\,3} (1) &= X^{\ast\,2}\otimes X^\ast\circ \Delta(1) = -X^{\ast\,2}(X)\otimes X^\ast(X) + \fa X^{\ast\,2}(1)\otimes X^\ast(X) = 2\fa,\\
X^{\ast\,3}(X) &= -X^{\ast\,2}(X^2)\otimes X^\ast(X) + \fa X^{\ast\,2}(X)\otimes X^\ast(X) = -\fb-\fa^2,\\
X^{\ast\,3}(X^2) &= -\fb X^{\ast\,2}(X)\otimes X^\ast(X) - \fc X^{\ast\,2}(1)\otimes X^\ast(X) =\fa\fb - \fc.
\end{align*}
We can express $X^{\ast\, 3}$ in terms of the basis $\varepsilon, X^\ast, X^{\ast\,2}$ and get
\[
X^{\ast\,3} = -2\fa X^{\ast\,2}+(\fb-\fa^2)X^\ast + (\fc+\fa\fb)\varepsilon.
\]
In particular, $A^\ast \cong \ZMVX / (f(X))$ with 
\[
f(X) = X^3 + 2\fa X^2 - (\fb-\fa^2) X - (\fc+\fa\fb).
\]
We also have that $\imath^\ast$ vanishes on $\varepsilon$ and $X^\ast$, while $X^{\ast\,2}$ is sent to $-1$. Hence $\imath$ can be identified with $\varepsilon$ from $\Fr_f$. Now consider $m^\ast(\varepsilon) = \varepsilon\circ m\colon A\otimes A \to \ZMV$. This vanishes on $1\otimes 1$, $1\otimes X$, and $X\otimes 1$, while each of $1\otimes X^2$, $X\otimes X$, and $X^2\otimes 1$ are sent to $-1$. Also, $X\otimes X^2$ and $X^2\otimes X$ are sent to $-\fa$, and $X^2\otimes X^2$ is sent to $-\fb-\fa^2$. It is straightforward to check that
\[
-\varepsilon\otimes X^{\ast \,2}- X^\ast\otimes X^\ast-X^{\ast\, 2}\otimes \varepsilon -2\fa(\varepsilon\otimes X^\ast+X^\ast\otimes \varepsilon)+(\fb-\fa^2)\varepsilon\otimes \varepsilon
\]
is the same map $A\otimes A \to \ZMV$. Therefore $\FrMV^\ast \cong \Fr_f$. By Lemma \ref{lm:easyiso} $\FrMV$ is self-dual via the isomorphism $\Phi_\fa$.
\end{proof}

This means that all base changes of $\FrMV$ are self-dual as well.

Let $\Fr = (\Lambda, A, \varepsilon, \Delta)$ be a Frobenius system and $x\in A$ be a unit. Then $\Fr^x = (\Lambda, A, \varepsilon^x, \Delta^x)$ given by
\[
\varepsilon^x(r) = \varepsilon(xr), \hspace{1cm} \Delta^x(a) = \Delta(x^{-1}a) = x^{-1}\Delta(a)
\]
for $r\in \Lambda$ and $a\in A$, is a Frobenius system, called the {\em twisting of $\Fr$ by $x$}.

Consider the Frobenius system $\Fr_f$ from Example \ref{ex:easy_change}, and let $\beta\in \Lambda$ be a unit. The ring isomorphism $\Psi_\beta\colon \Lambda[X]\to \Lambda[X]$ induced by $\Psi_\beta(X) = \beta X$ induces a ring isomorphism $\Psi_\beta\colon \Lambda[X]/(f(X)) \Lambda[X]/(g(X))$, where
\[
g(X) = X^3 - a\beta^{-1}X^2-b\beta^{-2}X - c\beta^{-3}.
\]
For this to induce an isomorphism on Frobenius systems, we need to twist one of the systems by $\beta^{-2}$.

\begin{lemma}\label{lm:twistbeta}
The ring isomorphism $\Psi_\beta\colon \Lambda[X]/(f(X)) \Lambda[X]/(g(X))$ induces an isomorphism $\Fr_f\cong \Fr_g^{\beta^{-2}}$.
\end{lemma}

\begin{proof}
It remains to show that $\varepsilon_f(X^k) = \varepsilon^{\beta^{-2}}(X^k)$ for $k=0,1,2$, and $\Psi_\beta\otimes \Psi_\beta (\Delta_f(1)) = \Delta_g^{\beta^{-2}}(1) = \beta^2\Delta_g(1)$. Both calculations are straightforward from the definitions.
\end{proof}

\section{Generalities on $\mathfrak{sl}_3$-link homology}
\label{sec:link_hom}

A {\em closed web} $\Gamma$ is a finite trivalent oriented graph in $\R^2$, possibly with vertex-less loops, such that at each vertex all edges are either incoming, or all are outgoing. 

An oriented link diagram $D$ with $n$ crossings gives rise to $2^n$ closed webs by resolving each crossing in two different ways as in Figure \ref{fig:resolutions}.

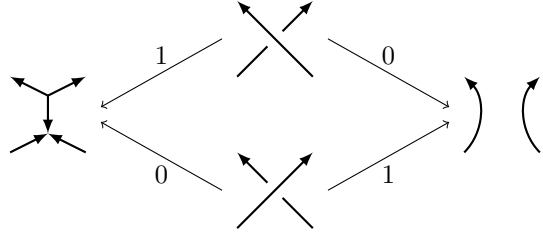
\begin{figure}[ht]
\begin{tikzpicture}
\pcrossing{0}{0}{1}
\mcrossing{0}{2}{1}
\smoothing{3}{1}{1}{->}
\spidering{-3}{1}{1}{->}
\draw[->] (1.2, 0.5) -- node [below] {$1$} (2.8, 1.4);
\draw[->] (1.2, 2.5) -- node [above] {$0$} (2.8, 1.6);
\draw[->] (-0.2, 0.5) -- node [below] {$0$} (-1.8, 1.4);
\draw[->] (-0.2, 2.5) -- node [above] {$1$} (-1.8, 1.6);
\end{tikzpicture}
\caption{\label{fig:resolutions}The $0$ and $1$ resolutions of a positive (the lower diagram) and negative (the upper diagram) crossing.}
\end{figure}

The {\em Kuperberg bracket} $\langle \Gamma\rangle$ of a web $\Gamma$ is the Laurent polynomial in one variable $q$ determined by the relations
\begin{align*}
\left\langle \Gamma \sqcup
\begin{tikzpicture}[baseline={([yshift=-.6ex]current bounding box.center)}]
\draw[thick] (0,0) circle (4pt);
\end{tikzpicture}
\right\rangle &= (q^2+1+q^{-2}) \langle \Gamma \rangle,\\
\left\langle 
\begin{tikzpicture}[baseline={([yshift=-.6ex]current bounding box.center)}]
\draw[>=latex, ->, thick] (0,0) -- (0.5,0);
\draw[>=latex, <-, thick] (0.5,0) to [out = 45, in = 135] (1.5, 0);
\draw[>=latex, <-, thick] (0.5,0) to [out = 315, in = 225] (1.5, 0);
\draw[>=latex,->, thick] (1.5,0) -- (2,0);
\end{tikzpicture}
\right\rangle &= (q+q^{-1}) \langle
\begin{tikzpicture}[baseline={([yshift=-.6ex]current bounding box.center)}]
\draw[>=latex, ->, thick] (0,0) -- (2,0);
\end{tikzpicture}
\rangle, \\
\left\langle 
\begin{tikzpicture}[baseline={([yshift=-.6ex]current bounding box.center)}]
\draw[>=latex,->, thick] (0,0) -- (0.25,0.25);
\draw[>=latex,->, thick] (0.75,0.25) -- (0.25,0.25);
\draw[>=latex,->, thick] (0.25,0.75) -- (0.25,0.25);
\draw[>=latex,->, thick] (0.75,0.25) -- (1, 0);
\draw[>=latex,->, thick] (0.75, 0.25) -- (0.75, 0.75);
\draw[>=latex,->, thick] (0.25, 0.75) -- (0, 1);
\draw[>=latex,->, thick] (0.25, 0.75) -- (0.75, 0.75);
\draw[>=latex,->, thick] (1, 1) -- (0.75, 0.75);
\end{tikzpicture}
\right\rangle &= \left\langle 
\begin{tikzpicture}[baseline={([yshift=-.6ex]current bounding box.center)}]
\draw[>=latex,->, thick] (0,0) to [out = 45, in = 315] (0,1);
\draw[>=latex,->, thick] (1,1) to [out = 225, in = 135] (1, 0);
\end{tikzpicture}
\right\rangle + \left\langle
\begin{tikzpicture}[baseline={([yshift=-.6ex]current bounding box.center)}]
\draw[>=latex,->, thick] (0,0) to [out = 45, in = 135] (1,0);
\draw[>=latex,->, thick] (1,1) to [out = 225, in = 315] (0,1);
\end{tikzpicture}
\right\rangle .
\end{align*}

In \cite{MR2100691}, Khovanov used foams and the rank 3 Frobenius system $\FrK$ to construct graded free abelian groups $\FrK(\Gamma)$ for a web $\Gamma$ such that the graded rank of this abelian group is $\langle \Gamma \rangle$. Moreover, given an oriented link diagram $D$ he used these groups to construct a $q$-graded cochain complex $\Csl(D;\FrK)$, whose bigraded homology groups, denoted here by $\Hsl^{i,j}(L;\FrK)$ only depend on the underlying oriented link $L$.

This construction was generalized by Mackaay and Vaz \cite{MR2336253} to any base change system $\Fr = (\Lambda, A, \varepsilon, \Delta)$ of $\FrMV$, so that $\Fr(\Gamma)$ is a graded free $\Lambda$-module of graded rank $\langle\Gamma\rangle$, and the resulting homology groups $\Hsl^{i,j}(L;\Fr)$ are link invariants.

To define $\Fr(\Gamma)$ and the $\mathfrak{sl}_3$-homology groups, we need the notion of a foam. We refer the reader to \cite{MR2100691} for a precise definition. Here we think of a foam as a surface that admits singular arcs, so that each point of a singular arc has a neighborhood homeomorphic to the letter $Y$ crossed with an interval. Removing the singular set we get a surface which is oriented in a way that induces an orientation on the singular arcs. We call the closures of the components of the non-singular set {\em facets}, so that each singular arc has three adjacent facets.

In particular, we think of foams as cobordisms between webs $\Gamma_0$, $\Gamma_1$ embedded in $\R^2\times [0,1]$, where $\Gamma_i$ is embedded in $\R^2\times \{i\}$ for $i=0,1$. The embedding of a foam into $\R^3$ gives rise to a cyclic ordering of the facets around a singular arc using the {\em left-hand rule}, see Figure \ref{fig:left-hand}.

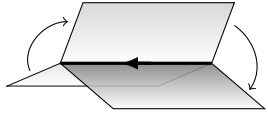
\begin{figure}[ht]
\begin{tikzpicture}
\draw[-] (0,0) -- (-0.7, -0.3) -- (0.3, -0.3) -- (1.3, -0.3) -- (2, 0);
\shade[color = gray, opacity = 0.3] (0,0) -- (0.3 ,0.8) -- (2.3,0.8) -- (2,0);
\shade[color = gray!40, opacity = 0.2] (0,0) -- (-0.7, -0.3) -- (1.3, -0.3) -- (2, 0);
\shade[color = gray!40, opacity = 0.85] (0,0) -- (0.7, -0.6) -- (2.7, -0.6) -- (2, 0);
\draw[<-, very thick, >=latex] (0.8,0) -- (2,0);
\draw[-, very thick] (0,0) -- (1.2,0);
\draw[-] (0,0) -- (0.3 ,0.8) -- (2.3,0.8) -- (2,0);
\draw[-] (0,0) -- (0.7, -0.6) -- (2.7, -0.6) -- (2, 0);
%\draw[-] (0,0) -- (-0.7, -0.3) -- (0.3, -0.3);
%\draw[-, dashed] (0.3, -0.3) -- (1.3, -0.3) -- (2, 0);
\draw[->] (2.3, 0.5) to [out = 325, in = 55] (2.5, -0.35);
\draw[->] (-0.4, -0.1) to [out =115, in = 180] (0.1, 0.55);
\end{tikzpicture}
\caption{\label{fig:left-hand}The cyclic ordering of facets adjacent to an oriented singular arc.}
\end{figure}

Finally, foams are allowed to have finitely many marked points away from the singular arcs, called {\em dots}. A dot can move freely on its facet, but cannot cross singular arcs. 

A notable closed foam is the {\em theta-foam}
\[
\Theta = \{ (x,y,z) \in \R^3 \mid x^2+y^2+z^2 = 1 \mbox{ or } (z = 0 \mbox{ and } x^2+y^2 \leq 1)\},
\]
which has one singular circle consisting of the equator $x^2+y^2 = 1$. A compatible orientation on $\Theta$ induces a cyclic ordering on the three facets. We assume the standard orientation on $\Theta$ to be such that the facet $\{(x,y,0)\mid x^2+y^2\leq 1\}$ is followed by the facet $\{(x,y,z)\mid x^2+y^2+z^2=1, \, z\geq 0\}$. 

As in \cite{MR2336253} we define $\Fm$ to be the category whose objects are closed webs and whose morphisms are $\ZMV$-linear combinations of isotopy classes of foams. 

Also, we need the following relations from \cite{MR2336253}, which reflect the usual geometric interpretation of Frobenius systems in terms of dotted surfaces, compare \cite[\S 3.1]{MR2100691}.
\begin{align}
\begin{tikzpicture}[baseline={([yshift=-.6ex]current bounding box.center)}]
\parallelbox{0}{0}{1}{0.66}
\node[scale = 0.75] at (0.6, 0.33) {$\bullet\bullet\bullet$};
\end{tikzpicture}
 &= \fa
 \begin{tikzpicture}[baseline={([yshift=-.6ex]current bounding box.center)}]
\parallelbox{0}{0}{1}{0.66}
\node[scale = 0.75] at (0.6, 0.33) {$\bullet\,\bullet$};
\end{tikzpicture}
+\fb
\begin{tikzpicture}[baseline={([yshift=-.6ex]current bounding box.center)}]
\parallelbox{0}{0}{1}{0.66}
\node[scale = 0.75] at (0.6, 0.33) {$\bullet$};
\end{tikzpicture}
+\fc
\begin{tikzpicture}[baseline={([yshift=-.6ex]current bounding box.center)}]
\parallelbox{0}{0}{1}{0.66}
\end{tikzpicture}
\tag{3D}\label{eq:3d}\\[0.2cm]
- \,\,
\begin{tikzpicture}[baseline={([yshift=-.6ex]current bounding box.center)}]
\cylinder{0}{0}{1}{1.2}
\end{tikzpicture}
\,&= 
\begin{tikzpicture}[baseline={([yshift=-.6ex]current bounding box.center)}]
\bowl{0}{1.2}{1}{0.05}
\bowlud{0}{0}{1}{0.05}
\node[scale = 0.75] at (0.4, 0.9) {$\bullet\,\bullet$};
\end{tikzpicture}
+
\begin{tikzpicture}[baseline={([yshift=-.6ex]current bounding box.center)}]
\bowl{0}{1.2}{1}{0.05}
\bowlud{0}{0}{1}{0.05}
\node[scale = 0.75] at (0.4, 0.9) {$\bullet$};
\node[scale = 0.75] at (0.4, 0.1) {$\bullet$};
\end{tikzpicture}
+
\begin{tikzpicture}[baseline={([yshift=-.6ex]current bounding box.center)}]
\bowl{0}{1.2}{1}{0.05}
\bowlud{0}{0}{1}{0.05}
\node[scale = 0.75] at (0.4, 0.1) {$\bullet\,\bullet$};
\end{tikzpicture}
-\fa \left(
\begin{tikzpicture}[baseline={([yshift=-.6ex]current bounding box.center)}]
\bowl{0}{1.2}{1}{0.05}
\bowlud{0}{0}{1}{0.05}
\node[scale = 0.75] at (0.4, 0.9) {$\bullet$};
\end{tikzpicture}
+
\begin{tikzpicture}[baseline={([yshift=-.6ex]current bounding box.center)}]
\bowl{0}{1.2}{1}{0.05}
\bowlud{0}{0}{1}{0.05}
\node[scale = 0.75] at (0.4, 0.1) {$\bullet$};
\end{tikzpicture}
\right) - \fb
\begin{tikzpicture}[baseline={([yshift=-.6ex]current bounding box.center)}]
\bowl{0}{1.2}{1}{0.05}
\bowlud{0}{0}{1}{0.05}
\end{tikzpicture}
\tag{CN} \label{eq:cn}\\[0.2cm]
\begin{tikzpicture}[baseline={([yshift=-.6ex]current bounding box.center)}]
\sphere{0}{0}{0.45}
\end{tikzpicture}
&= 
\begin{tikzpicture}[baseline={([yshift=-.6ex]current bounding box.center)}]
\sphere{0}{0}{0.45}
\node[scale = 0.75] at (0, 0.25) {$\bullet$};
\end{tikzpicture}
=0, \hspace{1cm}
\begin{tikzpicture}[baseline={([yshift=-.6ex]current bounding box.center)}]
\sphere{0}{0}{0.45}
\node[scale = 0.75] at (0, 0.25) {$\bullet\, \bullet$};
\end{tikzpicture}
= -1
\tag{S}\label{eq:s}
\end{align}
We also require a relation reflecting the evaluation on theta-foams. For non-negative integers $k,l,m$ let $\Theta(k,l,m)$ be the theta-foam with standard orientation, such that the facet with positive $z$-coordinate is dotted $k$-times, the facet with negative $z$-coordinate is dotted $l$-times, and the remaining facet is dotted $m$-times. If all $k,l,m\in \{0,1,2\}$, we set
\begin{equation}\label{eq:theta}\tag{$\Theta$}
\Theta(k,l,m) = \left\{
\begin{array}{cl}
1 & (k,l,m) = (1,2,0) \mbox{ or a cyclic permutation} \\
-1 & (k,l,m) = (2,1,0) \mbox{ or a cyclic permutation} \\
0 & \mbox{else}
\end{array}
\right.
\end{equation}
compare Figure \ref{fig:theta_example}.
\begin{figure}[ht]
\begin{tikzpicture}
\thetafoam{0}{0}{0.75}
\node at (0, 0.5) {$\bullet$};
\node at (0, -0.5) {$\bullet \, \bullet$};
\node at (1,0) {$=$};
\thetafoam{2}{0}{0.75}
\node at (2,0) {$\bullet\, \bullet$};
\node at (2, -0.5) {$\bullet$};
\node at (3,0) {$=$};
\thetafoam{4}{0}{0.75}
\node at (4,0.5) {$\bullet\,\bullet$};
\node at (4,0) {$\bullet$};
\node at (5.2,0) {$=1$};
\end{tikzpicture}
\caption{\label{fig:theta_example}Theta foams evaluated to $1$.}
\end{figure}
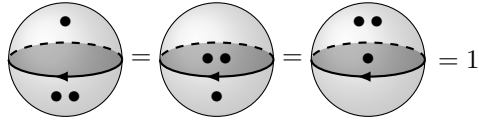

If $U$ is a closed foam, one can perform finitely many surgeries until it is a disjoint union of spheres and theta-foams. In particular, we can use the relations (\ref{eq:3d}), (\ref{eq:cn}), (\ref{eq:s}), and (\ref{eq:theta}) to assign an element $\FrMV(U)\in \ZMV$ which is well-defined by \cite{MR2100691, MR2336253}.

\begin{definition}
The category $\Fm_\ell$ is the quotient of the category $\Fm$ by the local relations (\ref{eq:3d}), (\ref{eq:cn}), (\ref{eq:s}), and (\ref{eq:theta}). For webs $\Gamma$, $\Gamma'$ we get
\[
\Hom_{\Fm_\ell}(\Gamma, \Gamma') = \Hom_{\Fm}(\Gamma, \Gamma')/\!\sim,
\]
where $\sum_i c_iU_i\sim 0$ (with $c_i\in \ZMV$ and $U_i$ foams from $\Gamma$ to $\Gamma'$) if and only if
\[
\sum_ic_i \FrMV(V'\circ U_i\circ V) = 0
\]
for all foams $V$ from $\emptyset$ to $\Gamma$ and $V'$ from $\Gamma'$ to $\emptyset$. The $\ZMV$-module $\FrMV(\Gamma)$ is defined to be
\[
\FrMV(\Gamma) = \Hom_{\Fm_\ell}(\emptyset, \Gamma).
\]
\end{definition}

As in \cite{MR2100691, MR2336253}, foams can be given a $q$-grading so that $\FrMV(\Gamma)$ is a graded $\ZMV$-module. This grading is given by
\[
|U|_q = -2\chi(U) + \chi(\partial U) + 2d(U),
\]
where $\chi$ is Euler characteristic of the underlying CW-complex, and $d(U)$ is the number of dots on $U$. As in \cite{MR2100691, MR2336253}, we have

\begin{proposition}
Let $\Gamma$ be a web. Then $\FrMV(\Gamma)$ is a graded free $\ZMV$-module of graded rank $\langle\Gamma\rangle$.\hfill\qed
\end{proposition}

If $\psi\colon \ZMV\to\Lambda$ is a ring homomorphism and $\Fr$ the base change system, we get an analogous category $\Fm^\psi_\ell$ where the morphism sets are $\Lambda$-modules, and a free $\Lambda$-module $\Fr(\Gamma)$ for every web. Furthermore, if $\Lambda$ is graded and $\psi$-grading preserving, the graded rank of $\Fr(\Gamma)$ is still $\langle\Gamma\rangle$.

Now let $D$ be an oriented link diagram with $n$ ordered crossings $c_1,\ldots, c_n$. Let $p_+$, respectively $p_-$, be the number of positive, respective negative, crossings. Each $J\in \{0,1\}^n$ leads to a web $D_J$ via the resolution rule in Figure \ref{fig:resolutions}. For $J = (j_1,\ldots,j_n)$ let $|J| = \sum_i j_i$. Given a grading preserving base change system $\Fr$ from $\FrMV$, we define a $q$-graded cochain complex $\Csl(D;\Fr)$ as follows. The cochain group in homological degree\footnote{We note here that our convention for a positive crossing is the opposite as used by Khovanov in \cite{MR2100691}. Nevertheless, the $0$-resolution for what we call a positive crossing agrees with the $0$-resolution of what Khovanov calls a negative crossing.} $i-p_+$ is
\[
\Csl^{i-p_+}(D;\Fr) = \bigoplus_{J, |J| = i} q^{3p_+-2p_--i}\Fr(D_J).
\]
The boundary map is induced by the unzip-foam for a positive crossing, and the zip-foam for a negative crossing, see Figure \ref{fig:zip-foams}.

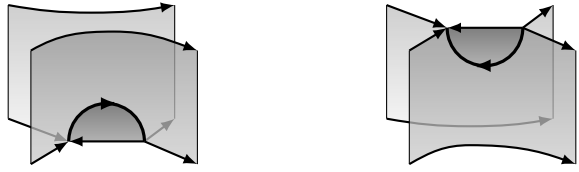
\begin{figure}[ht]
\begin{tikzpicture}
\shade[top color = gray!50, bottom color = gray!10, opacity = 0.75] (-0.3, 0.6) -- (0.5, 0.3) to [out = 90, in = 180] (1, 0.8) to [out = 0, in = 90] (1.5, 0.3) -- (1.9, 0.6) -- (1.9, 2.1) to [out = 190, in = 0] (0.8, 2) to [out = 180, in = 350] (-0.3, 2.1) -- (-0.3, 0.6);
\draw[-] (1.9, 0.6) -- (1.9, 2.1);
\draw[-] (-0.3, 0.6) -- (-0.3, 2.1);
\shade[top color = gray, bottom color = gray!60] (0.5, 0.3) to [out = 90, in = 180] (1, 0.8) to [out = 0, in = 90] (1.5, 0.3) -- (0.5, 0.3);
\draw[->, thick, >=latex] (-0.3, 0.6) -- (0.5, 0.3);
\draw[->, thick, >=latex] (1.5, 0.3) -- (1.9, 0.6);
\shade[top color = gray!80, bottom color = gray!40, opacity = 0.75] (0,0) -- (0.5, 0.3) to [out = 90, in = 180] (1, 0.8) to [out = 0, in = 90] (1.5, 0.3) -- (2.2, 0) -- (2.2, 1.5) to [out = 160, in = 0] (1.1, 1.75) to [out = 180, in = 30] (0, 1.5) -- (0,0);
\draw[-] (0,0) -- (0, 1.5);
\draw[-] (2.2, 0) -- (2.2, 1.5);
\draw[->, >=latex, thick] (0,0) -- (0.5, 0.3);
\draw[<-, >=latex, thick] (0.5, 0.3) -- (1.5, 0.3);
\draw[->, >=latex, thick] (1.5, 0.3) -- (2.2, 0);
\draw[-, very thick] (0.5, 0.3) to [out = 90, in = 180] (1, 0.8) to [out = 0, in = 90] (1.5, 0.3);
\draw[->, very thick, >=latex] (1.1,0.8) -- (1.15, 0.8);
\draw[->, >=latex, thick] (0, 1.5) to [out = 30, in = 180] (1.1, 1.75) to [out = 0, in = 160] (2.2, 1.5);
\draw[->, >= latex, thick] (-0.3, 2.1) to [out = 350, in = 180] (0.8, 2) to [out = 0, in = 190] (1.9, 2.1);
\shade[top color = gray!50, bottom color = gray!10, opacity = 0.75] (4.7, 2.1) -- (5.5, 1.8) to [out= 270, in = 180] (6, 1.3) to [out=0, in = 270] (6.5, 1.8) -- (6.9, 2.1) -- (6.9, 0.6) to [out = 190, in = 0] (5.8, 0.5) to [out = 180, in = 350] (4.7, 0.6) -- (4.7, 2.1);
\draw[->, >=latex, thick] (4.7, 0.6) to [out = 350, in = 180] (5.8, 0.5) to [out = 0, in = 190] (6.9, 0.6);
\draw[-] (4.7, 2.1) -- (4.7, 0.6);
\draw[-] (6.9, 0.6) -- (6.9, 2.1);
\shade[top color = gray, bottom color = gray!60] (5.5, 1.8) to [out = 270, in = 180] (6, 1.3) to [out = 0, in = 270] (6.5, 1.8);
\shade[top color = gray!80, bottom color = gray!40, opacity = 0.75] (5, 0) to [out = 30, in = 180] (5.8, 0.25) to [out = 0, in = 160] (7.2, 0) -- (7.2, 1.5) -- (6.5, 1.8) to [out = 270, in = 0] (6, 1.3) to [out = 180, in = 270] (5.5, 1.8) -- (5, 1.5) -- (5, 0);
\draw[->, >=latex, thick] (5, 0) to [out = 30, in = 180] (5.8, 0.25) to [out = 0, in = 160] (7.2, 0);
\draw[->, >=latex, thick] (5, 1.5) -- (5.5, 1.8);
\draw[<-, >=latex, thick] (5.5, 1.8) -- (6.5, 1.8);
\draw[->, >=latex, thick] (6.5, 1.8) -- (7.2, 1.5);
\draw[->, >=latex, thick] (6.5, 1.8) -- (6.9, 2.1);
\draw[->, >=latex, thick] (4.7, 2.1) -- (5.5, 1.8);
\draw[-, very thick] (5.5, 1.8) to [out = 270, in = 180] (6, 1.3) to [out = 0, in = 270] (6.5, 1.8);
\draw[<-, >=latex, very thick] (5.85, 1.3) -- (5.9, 1.3);
\draw[-] (5, 0) -- (5, 1.5);
\draw[-] (7.2, 0) -- (7.2, 1.5);
\end{tikzpicture}
\caption{\label{fig:zip-foams}The unzip-foam (left) associated to a positive crossing, and the zip-foam (right) associated to a negative crossing.}
\end{figure}

By \cite[Thm.2.4]{MR2336253} the corresponding homology groups $\Hsl(L;\Fr)$ are invariants of the underlying oriented link $L$. %Furthermore, by choosing a basepoint on $D$, we turn $\Csl(D;\Fr)$ into an $A$-cochain complex by letting $X$ act as placing a dot on the facet containing the basepoint. Passing to homology gives an invariant of the oriented based link.

If the base change $\psi\colon \ZMV \to \Lambda$ does not preserve the grading, we still get a cochain complex $\Csl(D;\Fr)$ and link invariants $\Hsl(L;\Fr)$ from this construction, but without $q$-grading.

Now, for a base change $\psi\colon \ZMV \to \Lambda$ write $a = \psi(\fa), b = \psi(\fb), c = \psi(\fc)\in \Lambda$ and let $f(X) \in \Lambda[X]$ be given by
\[
f(X) = X^3 - aX^2 - bX - c.
\]
Also, for $\alpha\in \Lambda$ let
\[
f'(X) = X^3 -(a-3\alpha)X^2-(b+2a\alpha-3\alpha^2)X-(c+b\alpha+a\alpha^2-\alpha^3)\in \Lambda[X].
\]
We have already seen in Lemma \ref{lm:easyiso} that the Frobenius systems $\Fr_f$ and $\Fr_{f'}$ are isomorphic with the isomorphism induced by $\Phi_\alpha\colon \Lambda[X]\to \Lambda[X]$ given by $\Phi_\alpha(X) = X+\alpha$. This isomorphism extends to the cochain complexes.

\begin{proposition}\label{prp:shift_equiv}
Let $D$ be an oriented link diagram. Then $\Phi_\alpha$ induces an isomorphism of cochain complexes $\Csl(D;\Fr_f)$ and $\Csl(D;\Fr_{f'})$. Moreover, if the base change $\psi\colon \ZMV \to \Lambda$ is grading preserving and $|\alpha| = 2$, this isomorphism is grading preserving.
\end{proposition}

\begin{proof}
Denote by $\Fm^\Lambda$ the category whose objects are closed webs and whose morphisms are $\Lambda$-linear combinations of isotopy classes of foams.  Let $\Gamma$, $\Gamma'$ be closed webs, and let $U$ be a foam from $\Gamma$ to $\Gamma'$. Let $T_U$ be the set of dots on $U$ and for $S\subset T_U$ let $U_S$ be the foam with dots given by $T_U-S$. In particular, $U = U_\emptyset$.

Define $\Phi_\alpha^{\Gamma,\Gamma'}\colon \Hom_{\Fm^\Lambda}(\Gamma, \Gamma')\to  \Hom_{\Fm^\Lambda}(\Gamma, \Gamma')$ by
\begin{equation}\label{eq:def_reln}
\Phi_\alpha^{\Gamma,\Gamma'}(U) = \sum_{S\subset T_U} \alpha^{|S|} U_S.
\end{equation}
First notice that that $\Phi_\alpha^{\Gamma,\Gamma'}$ sends the local relations using $(a,b,c)$ to local relations using $(a-3\alpha, b+2a\alpha - 3\alpha^2, c+b\alpha+a\alpha^2-\alpha^3)$. For example, for (\ref{eq:3d}) this just reflects the fact that $\Phi_\alpha\colon A\to A'$ is a well defined ring homomorphism, while (\ref{eq:cn}) holds since this $\Phi_\alpha$ preserves the co-multiplication. 

In particular, we have
\begin{equation}\label{eq:first_step}
\Fr_f(U) = \Fr_{f'}(\Phi_\alpha^{\emptyset, \emptyset}(U))
\end{equation}
for all closed foams $U$.

Also, if $V$ is a foam from $\Gamma'$ to the closed web $\Gamma''$, then
\begin{align}\label{eq:good_compo}
\begin{split}
\Phi_\alpha^{\Gamma',\Gamma''}(V)\circ \Phi_\alpha^{\Gamma, \Gamma'}(U) &= \sum_{S_1\subset T_V} \alpha^{|S_1|} V_{S_1} \circ \sum_{S_2\subset T_U}\alpha^{|S_2|} U_{S_2} \\
&= \sum_{S\subset T_{V\circ U}} \alpha^{|S|} (V\circ U)_S = \Phi_\alpha^{\Gamma, \Gamma''}(V\circ U).
\end{split}
\end{align}
Now assume that $\sum_ic_i U_i\in \Hom_{\Fm^\Lambda}(\emptyset, \Gamma)$ satisfies $\sum_ic_i\Fr_f(VU_i) = 0$ for any foam $V$ from $\Gamma$ to $\emptyset$. We claim that $\sum_ic_i \Phi_\alpha^{\emptyset, \Gamma}(U)$ represents $0$ in the quotient category corresponding to $f'(X)$, that is, we need $\sum_ic_i\Fr_{f'}(V\circ \Phi_\alpha^{\emptyset,\Gamma}(U)) = 0$.

By (\ref{eq:good_compo}), we have
\begin{equation}\label{eq:next_step}
V\circ \Phi_\alpha^{\emptyset, \Gamma}(U) = \Phi_\alpha^{\emptyset,\emptyset}(V\circ U) - \sum_{\emptyset \not= S_1\subset T_V} \alpha^{|S_1|} V_{S_1} \circ \Phi_\alpha^{\emptyset, \Gamma}(U).
\end{equation}
We now get $\sum_ic_i\Fr_{f'}(V\circ \Phi_\alpha^{\emptyset,\Gamma}(U)) = 0$ by an induction on the number of dots on $V$, using (\ref{eq:first_step}) and (\ref{eq:next_step}).

This shows that $\Phi_\alpha^{\emptyset, \Gamma}$ induces a $\Lambda$-linear map $\Fr_f(\Gamma)\to \Fr_{f'}(\Gamma)$. Using this map with $-\alpha$ shows that it is in fact an isomorphism.
\end{proof}

If $\beta\in \Lambda$ is a unit, Lemma \ref{lm:twistbeta} describes an isomorphism between $\Fr_f$ and a twisting of $\Fr_g$, where $g(X) = X^3 - a\beta^{-1}X^2-b\beta^{-2}X-c\beta^{-3}$, induced by $\Psi_\beta(X) = \beta X$. In the case of link-complexes we can compensate for the twisting.

\begin{proposition}\label{prp:scale_equiv}
Let $D$ be an oriented link diagram. Then $\Psi_\beta$ induces an isomorphism of cochain complexes $\Csl(D;\Fr_f)$ and $\Csl(D;\Fr_g)$. Moreover, if the base change $\psi\colon \ZMV \to \Lambda$ is grading preserving and $|\beta| = 0$, this isomorphism is grading preserving.
\end{proposition}

\begin{proof}
Let $\Gamma$ be a web and $U$ a foam between $\emptyset$ and $\Gamma$. Define
\[
\Psi_\beta^\Gamma(U) = \beta^{d(U)-\chi(U)} U.
\]
It is straightforward to see that $\Psi_\beta^\Gamma$ is compatible with the relations (\ref{eq:3d}), (\ref{eq:cn}), (\ref{eq:s}), and (\ref{eq:theta}) between $\Fr_f$ and $\Fr_g$. Hence this induces an isomorphism $\Psi_\beta^\Gamma\colon \Fr_f(\Gamma) \to \Fr_g(\Gamma)$.

Now notice that adding a zip-foam to a foam $U$ decreases the Euler characteristic by $1$, while adding an unzip-foam does not change the Euler characteristic. However, if $J\in \{0,1\}^n$, there is a well defined number $z(J)$ which is the number of negative crossings that have been resolved with $1$. We then define $\Psi_\beta\colon \Csl(D;\Fr_f)\to \Csl(D;\Fr_g)$ by using $\beta^{-z(J)}\Psi_\beta^{D_J}$ between the direct summands $\Fr_f(D_J)$ and $\Fr_g(D_J)$, which because of the extra factor commutes with the boundary.
\end{proof}

\section{Equivalence classes for polynomials}\label{sec:kr_equiv}

In this section we consider a field $\F$ and polynomials $f(X)\in \F[h, X]$ of the form
\begin{equation}\label{eq:stand_poly}
f(X) = X^3 - ahX^2 - bh^2X - ch^3,
\end{equation}
where $a,b,c\in \F$. The corresponding Frobenius system over $\F[h]$ is, as before, denoted by $\Fr_f$.

Lewark and Lobb \cite{MR3458146} introduced the notion of $\KR$-equivalence classes for such polynomials to analyze the corresponding link homologies.

\begin{definition}
Two polynomials $f, g$ of the form (\ref{eq:stand_poly}) are called {\em $\KR$-equivalent} over a link $L$, if $\Csl(D;\Fr_f)$ and $\Csl(D;\Fr_g)$ are chain homotopy equivalent as graded cochain complexes over $\F[h]$ for some link diagram $D$ of $L$. The polynomials are called {\em $\KR$-equivalent}, if they are $\KR$-equivalent for all links $L$.
\end{definition}

Proposition \ref{prp:shift_equiv} and Proposition \ref{prp:scale_equiv} are simple criteria to determine $\KR$-equivalence between some polynomials, compare \cite[Prop.3.3]{MR3458146}. Furthermore, Lewark and Lobb \cite[Thm. 3.7, Cor. 3.8]{MR3458146} show that over $\F = \C$ for a fixed link $L$ there exist only finitely many $\KR$-equivalence classes, and hence only countably many $\KR$-equivalence classes. Moreover, one of these classes is generic in that it is a countable intersection of Zariski-open sets.

We note that \cite{MR3458146} also considers $\mathfrak{sl}_n$-link homologies for all positive integers $n$.  In the case $n=3$ we are now going to complement their results slightly and also consider other characteristics for $\F$.

\begin{proposition}\label{prp:stand_forms}
Let $f$ be a polynomial of the form (\ref{eq:stand_poly}) and assume that $\F$ is algebraically closed.
\begin{enumerate}
\item Assume $\chr \F \not=3$. Then $f$ is $\KR$-equivalent to the polynomial
\begin{itemize}
\item $g(X) = X^3 - h^2X - wh^3$ for some $w\in \F$, or
\item $g(X) = X^3 - h^3$, or
\item $g(X) = X^3$.
\end{itemize}
\item Assume $\chr \F = 3$. Then $f$ is $\KR$-equivalent to the polynomial
\begin{itemize}
\item $g(X) = X^3 - hX^2 - wh^3$ for some $w\in \F$ (provided $a\not=0$), or
\item $g(X) = X^3 - h^2X$ (provided $a=0$ and $b\not=0$), or
\item $g(X) = X^3$ (provided $a=b=0$).
\end{itemize}
\end{enumerate}
\end{proposition}

\begin{proof}
If $\chr\F \not=3$, we can use Proposition \ref{prp:shift_equiv} to get the $X^2$-coefficient equal to $0$. If after that the $X$-coefficient is non-zero, we can normalize it to $-h^2$ using Proposition \ref{prp:scale_equiv}. This leads to the first case.

If, after making the $X^2$-coefficient equal to $0$, the $X$-coefficient is also $0$, we either have the third coefficient $0$ (leading to the third case) or we can normalize the third coefficient to $-h^3$ using Proposition \ref{prp:scale_equiv}.

If $\chr \F = 3$, we cannot make the $X^2$-coefficient equal to $0$ with Proposition \ref{prp:shift_equiv}. But if it is non-zero, we can normalize it to $-h$ using Proposition \ref{prp:scale_equiv}. After that we can use Proposition \ref{prp:shift_equiv} to make the $X$-coefficient equal to $0$. This leads to the first case.

If the $X^2$-coefficient is $0$ to begin with, we cannot change the $X$-coefficient with Proposition \ref{prp:shift_equiv}, but we can make the constant coefficient $0$. After that we are either in the third case, or use Proposition \ref{prp:scale_equiv} to get to the second case.
\end{proof}

\begin{remark}
Choosing the correct $\alpha$ or $\beta$ for Propositions \ref{prp:shift_equiv} and \ref{prp:scale_equiv} may have required us to solve a quadratic or cubic equation. Hence the assumption that $\F$ is algebraically closed in Proposition \ref{prp:stand_forms}. We now want to show that this is not a serious restriction.
\end{remark}

Since $\F$ is a field, the polynomial ring $\F[h]$ is Euclidean. In particular, we can use the standard Smith-Normal-Form algorithm on the finitely generated and free cochain complexes $\Csl(D;\Fr_f)$ over $\F[h]$. Since this can be done in a grading-preserving fashion, $\Csl(D;\Fr_f)$ is graded-chain homotopic to a direct sum of complexes of the form
\begin{itemize}
\item $C_\F = u^iq^j \F[h]$, and
\item $C_\F(k) = u^iq^j \F[h] \stackrel{h^k}{\longrightarrow} u^{i+1}q^{j+2k} \F[h]$ with $k\geq 1$.
\end{itemize}
Here $u^iq^j$ indicates the bi-degree which the copy of $\F[h]$ has in the cochain complex. %We call such complexes {\em elementary}.

So if $\varphi\colon \F \to \overline{\F}$ is a ring homomorphism between two fields, we get another base change Frobenius system that we denote $\Fr_f\otimes_\F \overline{\F}$. Then $\Csl(D;\Fr_f\otimes_\F \overline{\F}) = \Csl(D;\Fr_f)\otimes_\F \overline{\F}$, and this has the exact same Smith-Normal-Form decomposition as $\Csl(D;\Fr_f)$. Since the Smith-Normal-Form determines the graded chain homotopy type, we get the next lemma.

\begin{lemma}\label{lm:change_equiv}
Let $\varphi\colon \F\to \overline{\F}$ be a ring homomorphism between fields and $f, g$ polynomials of the form (\ref{eq:stand_poly}).  Then $f$ is $\KR$-equivalent to $g$ if and only if $\varphi(f)$ is $\KR$-equivalent to $\varphi(g)$.\hfill \qed
\end{lemma}

Since every field includes into an algebraically closed field, we can drop this requirement from Proposition \ref{prp:stand_forms}. It will still be convenient to assume that a field is algebraically closed. 
Indeed, let $\overline{\F}$ be algebraically closed, let $w\in \overline{\F}$ and $f(X) = X^3 - h^2X - wh^3$. First observe that $f$ is $\KR$-equivalent to $g(X) = X^3 - h^2X + wh^3$, as an application of Proposition \ref{prp:scale_equiv} with $\beta = -1$. It is also $\KR$-equivalent to $h(X) = X^3 - h^2X \pm \varphi(w)h^3$ for every ring automorphism $\varphi\colon \overline{\F}\to \overline{\F}$ by Lemma \ref{lm:change_equiv}.

Now let $\F$ be the prime field contained in $\overline{\F}$. Then $f$ can also be considered a polynomial over $\F(w)[h]$, where $\F(w)\subset \overline{\F}$ is the simple extension of $\F$ by $w$. Galois Theory provides us potentially with many ring homomorphisms $\varphi\colon \F(w)\to\overline{\F}$ such that $\varphi(w)\not = w$. Although for some $w$ we may only ever get $\varphi(w) = w$. Maybe the neatest result we get is when $\F(w)$ is a transcendental extension of $\F$. In that case we can send $w$ to any other transcendental $w'$.

\begin{corollary}
Let $w,w'\in \overline{\F}$ be transcendental. Then $X^3-h^2X -wh^3$ is $\KR$-equivalent to $X^3-h^2X-w'h^3$.\hfill \qed 
\end{corollary}

In particular, for $\overline{\F} = \C$ we get that the Lewark--Lobb generic equivalence class from \cite[Cor.3.8]{MR3458146} contains all polynomials $X^3-h^2X -wh^3$ with $w\in \C$ transcendental.

\section{$\mathfrak{sl}_3$-homology for separable polynomials}
\label{sec:hom_sep}

The purpose of this section is to slightly generalize \cite[\S 3.1]{MR2336253}. There, Mackaay and Vaz show that for $\Fr_f = (\C, \C[X]/(f(X)), \varepsilon,\Delta)$ with $f(X)$ a separable (that is, a polynomial with three distinct roots) polynomial of degree $3$, the homology $\Hsl(L;\Fr_f)$ has $\C$-dimension $3^n$, where $n$ is the number of components of $L$, and is concentrated in even homological degrees with at least $3$ generators in homological degree $0$.

The Frobenius system we want to look at here is going to be of the form $\Fr_f = (\Lambda, \Lambda[X]/(f(X)), \varepsilon, \Delta)$, where $\Lambda$ is a commutative ring with $1$, and $f(X) = (X-r_1)(X-r_2)(X-r_3)$ with $r_1,r_2,r_3\in \Lambda$. We also need to require the following condition on the roots of $f$.
\begin{equation}\label{eq:root_cond}
r_i-r_j \mbox{ is a unit in }\Lambda\mbox{ for }i\not=j.
\end{equation}
This condition forces the roots to be different, and if $\Lambda$ is a field, the condition is equivalent to $f$ being separable. The purpose of this condition is to use a general form of the Chinese Remainder Theorem, compare \cite[II.Thm.2.1]{MR1878556}.

We note that for $\Lambda = \K[h]$ and $f(X) = (X-r_1h)(X-r_2h)(X-r_3h)$, where $r_1,r_2,r_3\in \K$ this condition cannot be satisfied, since $(r_i-r_j)h$ is not going to be a unit. However, the base change to $h^{-1}\Lambda = \K[h,h^{-1}]$ will often satisfy (\ref{eq:root_cond}), in particular when $\K$ is a field and all three roots are different. This ring carries a $q$-grading respected by $h^{-1}\Fr_f$, and so $\Hsl(L;h^{-1}\Fr_f)$ is bigraded.

The analogue of \cite[Thm.3.11]{MR2336253} is

\begin{theorem}\label{thm:nice_homology}
Let $L$ be a link and $\Fr_f = (\Lambda, \Lambda[X]/(f(X)), \varepsilon, \Delta)$ a Frobenius system over a commutative ring $\Lambda$ with $f(X) = (X-r_1)(X-r_2)(X-r_3)$ such that (\ref{eq:root_cond}) is satisfied. Then $\Hsl(L;\Fr_f)$ is free over $\Lambda$ of rank $3^n$, where $n$ is the number of components. Moreover, all homology is concentrated in even degrees, with at least three copies of $\Lambda$ in homological degree $0$.
\end{theorem}

The proof mostly just follows \cite[\S 3.1]{MR2336253}, but there are a few steps where $\Lambda$ not being a field requires a more involved argument. We will also have use later for some of the constructions required.

As mentioned before, condition (\ref{eq:root_cond}) ensures that the Chinese Remainder Theorem applies, that is, we have an isomorphism of $\Lambda$-algebras
\[
\Lambda[X]/(f(X))\cong \Lambda[X]/(X-r_1)\oplus \Lambda[X]/(X-r_2)\oplus \Lambda[X]/(X-r_3).
\]
Moreover, the ring homomorphisms $\chi\colon \Lambda[X]/(X-r_i)\to \Lambda[X]/(f(X))$ are given by
\[
\chi(1) = I_{r_i}(X) = (X-r_j)(X-r_k)(r_i-r_j)^{-1}(r_i-r_k)^{-1},
\]
where $\{i,j,k\}=\{1,2,3\}$.

Let $\Gamma$ be a closed web and $E(\Gamma)$ be the set of edges in $\Gamma$ (this includes the loops). Following \cite[\S 6]{MR2100691} and \cite[Def.3.3]{MR2336253}, we associate to this web the commutative $\Lambda$-algebra $R(\Gamma)$ with generators $X_i$ for each $i\in E(\Gamma)$, modulo the relations
\begin{align}\label{eq:comm_relations}
\begin{split}
X_i+X_j+X_k &= r_1+r_2+r_3, \\
X_iX_j + X_jX_k+X_iX_k &= r_1r_2+r_2r_3+r_1r_3, \\
X_iX_jX_k &=r_1r_2r_3,
\end{split}
\end{align} 
for any triple of edges $i,j,k$ which share a trivalent vertex. If $i$ is a closed loop, we also require the relation
\begin{equation}\label{eq:spec_relation}
f(X_i) = (X_i-r_1)(X_i-r_2)(X_i-r_3) = 0.
\end{equation}
If $i$ is a non-closed edge, (\ref{eq:spec_relation}) follows from (\ref{eq:comm_relations}).

We now define {\em colorings} of closed webs, following \cite[Def.3.4]{MR2336253}. In \cite{MR2711478}, these were called states.

\begin{definition}
Let $\Gamma$ be a closed web. A {\em coloring} of $\Gamma$ is defined to be a function $\varphi\colon E(\Gamma)\to \{r_1,r_2,r_3\}$. The coloring is called {\em admissible}, if
\begin{align}\label{eq:admis_color}
\begin{split}
r_1+r_2+r_3 & = \varphi(i) + \varphi(j) + \varphi(k),\\
r_1r_2+r_2r_3+r_1r_3 & = \varphi(i)\varphi(j)+\varphi(j)\varphi(k)+\varphi(i)\varphi(k),\\
r_1r_2r_3 & = \varphi(i)\varphi(j)\varphi(k),
\end{split}
\end{align}
for any edges $i,j,k$ incident to the same trivalent vertex. The set of all colorings is denoted $C(\Gamma)$, with $AC(\Gamma)$ the subset of admissible colorings.
\end{definition}

Since (\ref{eq:spec_relation}) holds for all edges $i\in E(\Gamma)$, we have a homomorphism of $\Lambda$-algebras $\eta_i\colon \Lambda[X]/(f(X)) \to R(\Gamma)$ for each $i\in E(\Gamma)$ sending $X$ to $X_i$, and we define
\[
I_\varphi(\Gamma) = \prod_{i\in E(\Gamma)} \eta_i(I_{\varphi(i)}(X))\in R(\Gamma)
\]
for each coloring $\varphi\in C(\Gamma)$. As in \cite[Cor.3.6]{MR2336253}, we have
\begin{equation}\label{eq:sum_idem}
\sum_{\varphi\in C(\Gamma)} I_\varphi(\Gamma) = 1 \in R(\Gamma),
\end{equation}
and
\[
I_\varphi(\Gamma)I_\psi(\Gamma) = \left\{
\begin{array}{cc}
I_\varphi(\Gamma) & \varphi = \psi, \\
0 & \varphi\not= \psi.
\end{array}
\right.
\]

The next result is slightly more involved for general $\Lambda$ than the field version in \cite[Lm.3.7]{MR2336253}.

\begin{proposition}
Let $\Gamma$ be a closed web. Then there is an isomorphism of $\Lambda$-algebras
\[
R(\Gamma) \cong \bigoplus_{\varphi\in AC(\Gamma)}\Lambda,
\]
with the copy of $\Lambda$ corresponding to an admissible coloring $\varphi$ generated by $I_\varphi(\Gamma)$.
\end{proposition}

\begin{proof}
We first show that if the coloring $\varphi$ is not admissible, then $I_\varphi(\Gamma) = 0$. Note that $X_iI_\varphi(\Gamma) = \varphi(i)I_\varphi(\Gamma)$ for any coloring by the definition of $I_\varphi(\Gamma)$. Together with (\ref{eq:comm_relations}) this leads to
\begin{align}
(r_1+r_2+r_3) I_\varphi(\Gamma) &= (\varphi(i)+\varphi(j)+\varphi(k))I_\varphi(\Gamma), \label{eq:first_idem}\\
(r_1r_2+r_2r_3+r_1r_3)I_\varphi(\Gamma) & = (\varphi(i)\varphi(j)+\varphi(j)\varphi(k)+\varphi(i)\varphi(k)) I_\varphi(\Gamma), \label{eq:secon_idem}
\end{align}
whenever the edges $i,j,k$ have a trivalent vertex in common. If $\varphi$ is not admissible, then there exists such a triple $i,j,k$ with $\Phi = \{\varphi(i),\varphi(j),\varphi(k)\}\not= \{r_1,r_2,r_3\}$. Assume first that $\Phi$ contains two elements, say $r_1$ and $r_2$, with $r_1$ being hit twice. Then (\ref{eq:first_idem}) is $((r_1+r_2+r_3) I_\varphi(\Gamma) = (2r_1+r_2)I_\varphi(\Gamma)$, that is,
\[
(r_3-r_1)I_\varphi(\Gamma) = 0.
\]
But since $(r_3-r_1)$ is a unit, we have $I_\varphi(\Gamma) = 0$. 

If $\Phi$ only contains one element, say $r_1$, then (\ref{eq:first_idem}) reads
\begin{equation}\label{eq:sum_is_two}
(r_2+r_3)I_\varphi(\Gamma) = 2r_1 I_\varphi(\Gamma),
\end{equation}
while (\ref{eq:secon_idem}) together with (\ref{eq:sum_is_two}) implies
\begin{equation}\label{eq:prod_is_square}
r_2r_3 I_\varphi(\Gamma) = r_1^2 I_\varphi(\Gamma).
\end{equation}
For the unit $r=(r_2-r_1)(r_3-r_1)$ we get
\begin{align*}
rI_\varphi(\Gamma) & = (r_2r_3 - r_1(r_2+r_3)+r_1^2)I_\varphi(\Gamma)\\
&= (r_1^2 - 2r_1^2+r_1^2)I_\varphi(\Gamma) = 0,
\end{align*}
by using (\ref{eq:sum_is_two}) and (\ref{eq:prod_is_square}). Hence $I_\varphi(\Gamma) = 0$ also in this case. In particular, the sum in (\ref{eq:sum_idem}) need only be done over admissible colorings.

The rest of the proof works as in \cite[Lm.3.7]{MR2336253}.
\end{proof}

As in \cite{MR2100691} and \cite{MR2336253}, $R(\Gamma)$ acts on $\Fr_f(\Gamma)$: if $U$ is a foam, $X_i\cdot U$ puts a dot on the facet of $U$ that bounds the edge $i$. This action now leads to a decomposition as $\Lambda$-modules
\[
\Fr_f(\Gamma) \cong \bigoplus_{\varphi\in AC(\Gamma)} \Fr_f(\Gamma)_\varphi,
\]
where $\Fr_f(\Gamma)_\varphi = I_\varphi(\Gamma)\cdot \Fr_f(\Gamma)$. Furthermore, the proof of \cite[Lm.3.10]{MR2336253} carries over to show that for an admissible coloring $\varphi$
\[
\Fr_f(\Gamma)_\varphi \cong \Lambda
\]
as a $\Lambda$-module.

If $D$ is an oriented link diagram, let $E(D)$ be the set of edges in the graph obtained from $D$ by using crossings as vertices and forgetting whether the arcs in the diagram are overpasses or underpasses.  A {\em coloring} of $D$ is then a function $\varphi\colon E(D)\to \{r_1,r_2,r_3\}$. Such a coloring may or may not give rise to colorings of resolutions $\Gamma$ of $D$, compare Figure \ref{fig:color_link}.

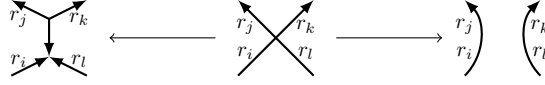
\begin{figure}[ht]
\begin{tikzpicture}
\jcrossing{0}{0}{1}
\smoothing{3}{0}{1}{->}
\spidering{-3}{0}{1}{->}
\draw[->] (1.3, 0.5) -- (2.7, 0.5);
\draw[->] (-0.3, 0.5) --  (-1.7, 0.5);
\node[scale = 0.8] at (0.1, 0.3) {$r_i$};
\node[scale = 0.8]  at (0.1, 0.65) {$r_j$};
\node[scale = 0.8] at (0.9, 0.3) {$r_l$};
\node[scale = 0.8] at (0.9, 0.65) {$r_k$};
\node[scale = 0.8] at (3, 0.3) {$r_i$};
\node[scale = 0.8]  at (3, 0.65) {$r_j$};
\node[scale = 0.8] at (4, 0.3) {$r_l$};
\node[scale = 0.8] at (4, 0.65) {$r_k$};
\node[scale = 0.8] at (-2.9, 0.2) {$r_i$};
\node[scale = 0.8]  at (-2.9, 0.75) {$r_j$};
\node[scale = 0.8] at (-2.1, 0.2) {$r_l$};
\node[scale = 0.8] at (-2.1, 0.75) {$r_k$};
\end{tikzpicture}
\caption{\label{fig:color_link}To get a coloring of the right resolution, we require $r_i=r_j$ and $r_k = r_l$. To get an admissible coloring on the left resolution, we need $\{r_i,r_l\} = \{r_j, r_k\}$ to be a two-element set.}
\end{figure} 

As in Section \ref{sec:link_hom} we write $D_J$ for $J\subset\{0,1\}^n$ for the resolutions of $D$, depending on an ordering of the crossings. 

\begin{definition}
A coloring $\varphi$ of $D$ is called {\em admissible}, if it admits an admissible coloring of at least one resolution $D_J$. An admissible coloring $\varphi$ of $D$ is called {\em canonical}, if the edges belonging to the same component of the underlying link have the same value. We denote the set of canonical colorings of $D$ by $C_c(D)$, and the set of admissible colorings by $AC(D)$.
\end{definition}

\begin{remark}\mbox{ }\label{rm:link_color}
\begin{enumerate}
\item If the coloring of $D$ is admissible, the admissible coloring of any $D_J$ is uniquely determined. Also, if $\varphi$ is a coloring of the web $D_J$, there is a unique coloring of $D$ corresponding to $\varphi$.
\item If the link has $n$ components, there are $3^n$ colorings $\varphi\colon E(D) \to \{r_1,r_2,r_3\}$ which are constant on components. Each of these colorings has exactly one resolution $D_J$ where it admits an admissible coloring of the web: in Figure \ref{fig:color_link} we have $r_i=r_k$ and $r_l = r_j$. If $r_i = r_j$ we have to resolve to the right to get an admissible coloring, and if $r_i\not= r_j$, we have to resolve to the left. In particular, all of these colorings are canonical. Furthermore, we have to resolve to the left an even number of times, so that the resolution $D_J$ is in even homological degree. A more precise statement can be obtained using linking numbers of components with different color, compare \cite[Thm.3.11]{MR2336253}.
\end{enumerate}
\end{remark}

Given an admissible coloring $\varphi\in AC(D)$, we can form
\[
\Csl(D;\Fr_f)_\varphi = \bigoplus_J q^{3p_+-2p_--|J|}\Fr_f(D_J)_\varphi,
\]
where the sum is over all $J$ such that $D_J$ admits an admissible coloring coming from $\varphi$. As in \cite[\S 3.1]{MR2336253}, this is a subcomplex of $\Csl(D;\Fr_f)$, and
\[
\Csl(D;\Fr_f) \cong \bigoplus_{\varphi\in AC(D)} \Csl(D;\Fr_f)_\varphi,
\]
as a direct sum of cochain complexes. To show Theorem \ref{thm:nice_homology} it remains to show that $\Csl(D;\Fr_f)_\varphi$ is contractible for $\varphi\in AC(D)-C_c(D)$ and isomorphic to one copy of $\Lambda$ for $\varphi\in C_c(D)$. But this follows as in the proof of \cite[Thm.3.9]{MR2336253}.

\section{Concordance invariants from $\mathfrak{sl}_3$-homology}

In this section we assume that $\F$ is a field, $f(X)=X^3-ahX^2-bh^2X-ch^3$ a polynomial with $a,b,c\in \F$ which has three different roots when viewed over $\overline{\F}[h]$, where $\overline{\F}$ is the algebraic closure of $\F$. Let $\Fr_f = (\F[h], \F[h,X]/(f(X)), \varepsilon, \Delta)$ be the corresponding Frobenius system, and $h^{-1}\Fr_f$ the base change system corresponding to the inclusion $\F[h]\subset \F[h,h^{-1}]$.

Let $K$ be a knot. By Theorem \ref{thm:nice_homology} $\Hsl^0(K;h^{-1}\overline{\Fr}_f)\cong (\overline{\F}[h,h^{-1}])^3$ as a graded module. Since localization is flat, $\F[h]$ is a Euclidean ring, and the inclusion $\F[h]\subset \overline{\F}[h]$ is also flat, we get an isomorphism of graded $\F[h]$-modules
\begin{equation}\label{eq:def_s_inv}
\Hsl^0(K;\Fr_f) \cong q^{s'}\F[h]\oplus q^{s''}\F[h]\oplus q^{s'''}\F[h]\oplus T,
\end{equation}
where $s'\leq s''\leq s'''$ are even integers, and $T$ is a torsion module. In particular, the $s',s'',s'''$ do not change if we pass to a field inclusion $\F\subset \overline{\F}$.

Observe that $\Hsl^{0,j}(K;h^{-1}\Fr_f)\cong \F^3$ for every even integer $j$. Now define
\begin{align*}
s'_f(K) &= \min\{j\in 2\Z\mid \dim_\F(\im (\Hsl^{0,j}(K;\Fr_f)\to \Hsl^{0,j}(K;h^{-1}\Fr_f))) \geq 1\} + 2,\\
s''_f(K) &= \min\{j\in 2\Z\mid \dim_\F(\im (\Hsl^{0,j}(K;\Fr_f)\to \Hsl^{0,j}(K;h^{-1}\Fr_f))) \geq 2\} , \\
s'''_f(K) &= \min\{j\in 2\Z\mid \dim_\F(\im (\Hsl^{0,j}(K;\Fr_f)\to \Hsl^{0,j}(K;h^{-1}\Fr_f))) \geq 3\} - 2.
\end{align*}

From (\ref{eq:def_s_inv}) we get $s'_f(K) = s'+2$, $s''_f(K) = s''$, and $s'''_f(K) = s'''-2$. We note that the shifts in the definition of $s'_f(K)$ and $s'''_f(K)$ are there to ensure that all three numbers vanish for the unknot. In analogy with the $s$-invariant from Khovanov homology one may expect that for a given knot these three numbers agree, but Lewark--Lobb \cite{MR3458146} have given examples where this is not the case, the smallest one being $10_{125}$ from Rolfsen's table.

These numbers are concordance invariants of $K$ and give lower bounds to the slice genus via
\[
|s^\ast_f(K)/4| \leq g_4(K), 
\]
see \cite{MR2554935, MR2509322}. Both papers work for general $\mathfrak{sl}_n$-link homology over $\C$. A proof along the line of Rasmussen's original argument \cite{MR2729272} can be done using the functoriality statement of \cite[\S 2.3]{MR2336253} while keeping track of the canonical generators along the cobordism.

As we already mentioned, the three numbers $s'_f(K), s''_f(K)$, and $s'''_f(K)$ need not agree. We will see later that they often do not agree for all $K$ except for one $f$. But in each characteristic there is a polynomial for which they do agree. 

Let $\chr \F \not= 3$ and let $f(X) = X^3 - h^3 = (X-h)(X^2+hX+h^2)$. For the polynomial $g(X) = (X^2+hX+h^2)$ we have $g(h) = 3h^2\not=0$, so $h$ is a simple root of $f$. Moreover, $g(X) = (X-ah)^2 = X^2 -2ahX+a^2h^2$ for some $a\in \F$ requires $a^2 = 1$, that is $a=\pm 1$. Since we ruled out characteristic $3$, this is not possible.

Hence $f(X)$ is separable over any algebraically closed field $\overline{\F}$ of characteristic different from $3$. While $f(X)$ is defined over any prime field, we will assume that there is $\rho\in \F$ with
\[
\rho^2+\rho+1 = 0.
\]
By the discussion above, this will not change the invariant. Recall the idempotents in $\F[h,X]/(X^3-h^3)$
\begin{align*}
I_1(X) &= (X-\rho h)(X-\rho^2 h) (h-\rho h)^{-1}(h-\rho^2 h)^{-1} = \frac{h^{-2}}{3} (X-\rho h)(X-\rho^2 h),\\
I_\rho(X) &=  (X-h)(X-\rho^2 h) (\rho h-h)^{-1}(\rho h-\rho^2 h)^{-1} = \frac{\rho h^{-2}}{3} (X-h)(X-\rho^2 h),\\
I_{\rho^2}(X) &=  (X-h)(X-\rho h) (\rho^2 h-h)^{-1}(\rho^2 h-\rho h)^{-1} = \frac{\rho^2 h^{-2}}{3} (X-h)(X-\rho h).
\end{align*}

Then $H^0(D; h^{-1}\Fr_{X^3-h^3})$ is freely generated by the elements
\begin{align*}
I_1 &= I_1(D_O) = I_1(X_1)\cdots I_1(X_k),\\
I_\rho &= I_\rho(D_O) = I_\rho(X_1)\cdots I_\rho(X_k),\\
I_{\rho^2} & = I_{\rho^2}(D_O) = I_{\rho^2}(X_1)\cdots I_{\rho^2}(X_k),
\end{align*}
where $D_O$ is the resolution of $D$ consisting of loops only, with $k$ loops. Note that $h^{-1}\Fr_{X^3-h^3}(D_O)\cong (\F[h,h^{-1},X]/(X^3-h^3))^k$, so $I_1,I_\rho, I_{\rho^2}$ are cochains, which by Section \ref{sec:hom_sep} are cocycles that generate $\Hsl^0(K;h^{-1}\Fr_{X^3-h^3})$.

Recall the ring automorphism $\Psi_\rho\colon \F[h, X]/(X^3-h^3) \to \F[h, X]/(X^3-h^3)$ with $\Psi_\rho(X) = \rho X$. By Proposition \ref{prp:scale_equiv} this induces a grading preserving automorphism of the cochain complexes $\Csl(D;\Fr_{X^3-h^3})$ and $\Csl(D;h^{-1}\Fr_{X^3-h^3})$, which we also denote by $\Psi_\rho$. A quick computation shows
\[
\Psi_\rho(I_1(X)) = I_{\rho^2}(X), \Psi_\rho(I_\rho(X)) = I_1, \Psi_\rho(I_{\rho^2}) = I_\rho,
\]
and therefore $\Psi_\rho$ permutes the generators $I_1, I_\rho$ and $I_{\rho^2}$ in the same way. Furthermore, $\Psi_\rho\colon \Hsl^0(K;h^{-1}\Fr_{X^3-h^3}) \to \Hsl^0(K;h^{-1}\Fr_{X^3-h^3})$ has three eigenvectors
\begin{align*}
E_1 & = I_1 + I_\rho + I_{\rho^2}, & \Psi_\rho(E_1) &= E_1,\\
E_\rho & = I_1 + \rho I_\rho + \rho^2 I_{\rho^2}, & \Psi_\rho(E_\rho) &= \rho E_\rho, \\
E_{\rho^2} &= I_1 + \rho^2 I_\rho + \rho I_{\rho^2}, & \Psi_\rho(E_{\rho^2}) &= \rho^2 E_{\rho^2}.
\end{align*}

Notice that there is a free $\F[h^3]$-cochain complex $C\subset \Csl(D;\Fr_{X^3-h^3})$ with
\[
\Csl(D;\Fr_{X^3-h^3}) \cong C \oplus hC \oplus h^2C
\]
as $\F[h^3]$-cochain complexes, where $hC = \{ h\otimes c \in \F[h]\otimes_{\F[h^3]}C\mid c \in C\}$ and similarly for $h^2C$.  In particular, $C\cong hC \cong h^2C$ as $\F[h^3]$-cochain complexes. Furthermore, we get the analogous decomposition
\[
\Csl(D;h^{-1}\Fr_{X^3-h^3}) \cong \overline{C}\oplus h\overline{C} \oplus h^2\overline{C},
\]
as $\F[h^3,h^{-3}]$-cochain complexes, with $\overline{C} = C\otimes_{\F[h^3]}\F[h^3,h^{-3}]$. The automorphism $\Psi_\rho$ keeps $C$, and therefore also $\overline{C}$ invariant. Now
\[
\Hsl^{0,2j}(K;h^{-1}\Fr_{X^3-h^3}) \cong \F^3 \cong H^{0,2j}(\overline{C}) \oplus H^{0,2j}(h\overline{C})\oplus H^{0,2j}(h^2\overline{C})
\]
and the latter three cohomology groups are all $\F$. Since $\Psi_\rho$ preserves $H^{0,2j}(\overline{C})$, and we have this is an eigenspace. In particular, $h^k E_{\rho^i}\in H^{0,2j}(\overline{C})$ for some integer $k$ and $i\in \{0,1,2\}$. Similarly, $H^{0,2j}(h\overline{C})$ and $H^{0,2j}(h^2\overline{C})$ are also eigenspaces of $\Psi_\rho$.

Consider the diagram
\[
\begin{tikzpicture}[scale = 0.825, transform shape]
\node at (0,1.5) {$\Hsl^{0,j}(K;\Fr)$};
\node at (4.2,1.5) {$H^{0,j}(C) \oplus H^{0,j}(hC) \oplus H^{0,j}(h^2C)$};
\node at (10,1.5) {$H^{0,j}(C) \oplus H^{0,j-2}(C) \oplus H^{0, j-4}(C)$};
\node at (1.25, 1.5) {$\cong$};
\node at (7.1, 1.5) {$\cong$};
\node at (-0.4,0) {$\Hsl^{0,j}(K;h^{-1}\Fr)$};
\node at (4.2,0) {$H^{0,j}(\overline{C}) \oplus H^{0,j}(h\overline{C}) \oplus H^{0,j}(h^2\overline{C})$};
\node at (10,0) {$H^{0,j}(\overline{C}) \oplus H^{0,j-2}(\overline{C}) \oplus H^{0, j-4}(\overline{C})$};
\node at (1.25, 0) {$\cong$};
\node at (7.1, 0) {$\cong$};
\draw[->] (0, 1.2) -- node [right] {$i$} (0, 0.3);
\draw[->] (2.3, 1.2) -- (2.3, 0.3);
\draw[->] (4, 1.2) -- (4, 0.3);
\draw[->] (5.9, 1.2) -- (5.9, 0.3);
\draw[->] (8, 1.2) -- (8, 0.3);
\draw[->] (9.8, 1.2) -- (9.8, 0.3);
\draw[->] (11.7, 1.2) -- (11.7, 0.3);
\end{tikzpicture}
\]

There exists a minimal $j$ such that $i\colon H^{0,j}(C)\to H^{0,j}(\overline{C}) \cong \F$ is surjective. The diagram shows that this minimal $j = s'$ from (\ref{eq:def_s_inv}), and 
\[
\im i\colon \Hsl^{0,j}(K;\Fr_{X^3-h^3}) \to \Hsl^{0,j}(K;h^{-1}\Fr_{X^3-h^3})
\]
is $1$-dimensional for $j = s'$, and at most $2$-dimensional for $j = s'+2$. In particular, there exists $a\in \Hsl^{0,s'}(K;\Fr_{X^3-h^3})$ with $i(a) = h^k E_{\rho^m}$ with $k\in \Z$ and $m\in \{0,1,2\}$ (the purpose of $k$ is to move the eigenvector into the correct $q$-degree).

Choosing a basepoint on the knot diagram $D$ gives an action of $\F[h,X]/(X^3-h^3)$ on $\Csl(D;\Fr_{X^3-h^3})$, where $X$ acts on a foam by placing a dot on the facet that bounds the arc with the basepoint. This action carries over to the localized version and commutes with the base change.

Then
\[
0 = (X-h)I_1, \hspace{1cm} 0 = (X-\rho h) I_\rho, \hspace{1cm}0 = (X-\rho^2 h) I_{\rho^2}.
\]
From this it follows that
\[
X E_1 = h E_\rho, \hspace{1cm} X E_\rho = h E_{\rho^2} \hspace{1cm} X E_{\rho^2} = h E_1,
\]
and therefore $i(Xa) = h^{k+1} E_{\rho^{m+1}}$ and $i(X^2 a) = h^{k+2} E_{\rho^{m-1}}$. Hence $s'' = s'+2$ and $s''' = s'+4$. We therefore define the $\mathfrak{sl}_3$-$s$-invariant as
\[
s^p_{\mathfrak{sl}_3}(K) = s'_{X^3-h^3}(K)/2 = s''_{X^3-h^3}(K)/2 = s'''_{X^3-h^3}(K)/2,
\]
where $p = \chr \F \not = 3$. We divide by $2$ to get the invariant closer to Rasmussen's original $s$-invariant.

Now assume that $\chr \F = 3$. We focus on the polynomial $f(X) = X^3 - h^2X = X(X-h)(X+h)$, which already factors over the prime field $\F_3$. The relevant idempotents are
\begin{align*}
I_0(X) &= -h^{-2} (X^2-h^2), \\
I_1(X) &= -h^{-2} (X^2+hX), \\
I_{-1}(X) &= -h^{-2} (X^2 - hX),
\end{align*}
giving rise to the generators $I_0, I_1, I_{-1}$ of $\Hsl^0(K;h^{-1}\Fr)$, where we drop the polynomial from the Frobenius system $\Fr$.

By Proposition \ref{prp:shift_equiv} there is a grading preserving automorphism
\[
\Phi_h\colon \Csl(D;\Fr)\to  \Csl(D;\Fr),
\]
which, after localization, permutes the generators
\[
\Phi_h(I_0) = I_{-1}, \hspace{1cm} \Phi_h(I_1) = I_0, \hspace{1cm} \Phi_h(I_{-1}) = I_1.
\]
This time, there is only one eigenvector $E_1 = I_0+I_1+I_{-1}$, but together with $J_1 = I_1 - I_{-1}$ and $J_2 = I_{-1}$ we get a basis in Jordan Normal Form
\[
\Phi_h(J_2) = J_2 + J_1, \hspace{1cm} \Phi_h(J_1) = J_1 + E_1, \hspace{1cm} \Phi_h(E_1) = E_1.
\]
In particular, the only $1$-dimensional subspace of $\Hsl^{0,j}(K;h^{-1}\Fr)$ invariant under $\Phi_h$ is the eigenspace of $E_1$, while the only $2$-dimensional $\Phi_h$-invariant subspace is generated by $E_1$ and $J_1$.

The definition of $\Phi_h$ given in (\ref{eq:def_reln}) shows that
\begin{equation}\label{eq:nice_shift}
\Phi_h(x) - x = h\, T(x),
\end{equation}
for all $x\in \Csl(D;\Fr)$, and some $T\colon \Csl(D;\Fr)\to \Csl(D;\Fr)$ lowering $q$-degree by $2$. Since multiplication by $h$ is injective and $\Phi_h-\id$ is a cochain map, $T$ also commutes with the boundary.

Let $s' = j$ be the minimal $q$-degree for which $i\colon \Hsl^{0,j}(K;\Fr)\to \Hsl^{0,j}(K;h^{-1}\Fr)$ is non-zero, and let $a\in \Hsl^{0,s'}(K;\Fr)$ be such that $i(a)\not=0$. Then $i(a)$ is an eigenvector, for otherwise $i(T(a)) \in \Hsl^{0,s'-2}(K;h^{-1}\Fr)$ satisfies
\[
i(T(a)) = h^{-1} (\Phi(i(a)) - i(a)) \not= 0,
\]
by (\ref{eq:nice_shift}), contradicting the minimality of $s'$. In particular, the image of $i$ in this $q$-degree is $1$-dimensional. Now note that $X$ acts on $\Csl(D;\Fr)$ and $\Csl(D;h^{-1}\Fr)$ as before, and
\[
X I_0 = 0, \hspace{1cm} X I_1 = h I_1, \hspace{1cm} X I_{-1} = -h I_{-1},
\]
so that
\[
X E_1 = h J_1 \mbox{ and } X (E_1-J_1) = h J_2.
\]
Hence the image of $i\colon \Hsl^{0,j}(K;\Fr)\to \Hsl^{0,j}(K;h^{-1}\Fr)$ is at least $2$-dimensional for $j = s'+2$, and $3$-dimensional for $j = s'+4$. If $i$ was also surjective for $j = s'+2$, we would use that $T(h^k J_{-1}) = J_1$, showing that the image for $j = s'$ would also be at least $2$-dimensional. Hence $s'_{X^3-h^2X}(K) = s''_{X^3-h^2X}(K) = s'''_{X^3-h^2X}(K)$ and we define
\[
s^3_{\mathfrak{sl}_3}(K) = s'_{X^3-h^2X}(K) / 2.
\]

\begin{remark}
The proof that $s'_f(K) = s''_f(K) = s'''_f(K)$ for $f(X) = X^3-h^3$ is mainly along the lines of Rasmussen's argument \cite[\S 3]{MR2729272}, compare also \cite[Thm.1.3]{MR2916277}, while the argument for $f(X) = X^3-h^2X$ in characteristic $3$ is similar to the one used in \cite[\S 2]{MR3189434} in characteristic $2$.
\end{remark}

It is shown in \cite[Thm.1.7]{MR2916277} that $s^0_{\mathfrak{sl}_3}$ induces a homomorphism from the knot-concordance group to the integers. We show that this also works in other characteristics.

\begin{theorem}\label{thm:s_hom}
Let $p$ be a prime number or $0$. Then the $\mathfrak{sl}_3$-s-invariant induces a homomorphism $s^p_{\mathfrak{sl}_3}\colon \mathfrak{C}\to \Z$ from the knot-concordance group $\mathfrak{C}$ to the integers.
\end{theorem}

\begin{proof}
Let $K_1$ and $K_2$ be knots, $s_i = s^p_{\mathfrak{sl}_3}(K_i)$ for $i=1,2$, and $f(X)$ be the polynomial $X^3-h^3$, if $p\not=3$, and $X^3-h^2X$, if $p = 3$. %Also, let $\F$ be the prime field of characteristic $p$. By (\ref{eq:def_s_inv}) we have
%\[
%\Hsl^0(K_i;\Fr_f) \cong q^{s_i-2} \F[h]\oplus q^{s_i} \F[h] \oplus q^{s_i+2}\F[h]\oplus T_i,
%\]
%with $T_i$ a $\F[h]$ a torsion module, for $i=1,2$. 

If we denote the split union by $K_1\sqcup K_2$, we have an isomorphism of cochain complexes
\[
\Csl(D_1\sqcup D_2; \Fr_f) \cong \Csl(D_1;\Fr_f)\otimes_{\F[h]}\Csl(D_2;\Fr_f),
\]
and surgery implies a commutative diagram
\[
\begin{tikzpicture}
\node at (0,1.5) {$\Hsl^{0,j_1}(K_1;\Fr_f)\otimes \Hsl^{0,j_2}(K_2; \Fr_f)$};
\node at (0,0) {$\Hsl^{0,j_1}(K_1;h^{-1}\Fr_f)\otimes \Hsl^{0,j_2}(K_2; h^{-1}\Fr_f)$};
\node at (7, 1.5) {$\Hsl^{0,j_1+j_2+2}(K_1\# K_2; \Fr_f)$};
\node at (7, 0) {$\Hsl^{0,j_1+j_2+2}(K_1\# K_2; h^{-1}\Fr_f)$};
\draw[->] (0, 1.2) -- (0, 0.3);
\draw[->] (7, 1.2) -- (7, 0.3);
\draw[->] (2.6, 1.5) -- node [above] {$S$}  (5, 1.5);
\draw[->] (3.2, 0) -- node [above] {$S$} (4.7, 0);
\end{tikzpicture}
\]
In the bottom row, applying $S$ to the generators $I_r$ for $r$ a root of $f$ gives $S(I_r\otimes I_r) = I_r$, and $S(I_r\otimes I_{r'}) = 0$ for $r\not= r'$, as surgery is just multiplication on the Frobenius system for these generators. Hence eigenvectors are mapped to eigenvectors, and the vertical map on the right is therefore non-zero for $j_1 = s_1-2$ and $j_2 = s_2-2$, that is, for $j_1+j_2+2 = s_1+s_2-2$. This shows that
\begin{equation}\label{eq:hom_inequ}
s^p_{\mathfrak{sl}_3}(K_1\# K_2) \leq s^p_{\mathfrak{sl}_3}(K_1) + s^p_{\mathfrak{sl}_3}(K_1).
\end{equation}
By Proposition \ref{prp:self_dual} and (\ref{eq:def_s_inv}) we have $s^p_{\mathfrak{sl}_3}(-K) = -s^p_{\mathfrak{sl}_3}(K)$ for any knot $K$ with $-K$ the mirror knot. Hence (\ref{eq:hom_inequ}) is actually an equality.
\end{proof}

In \cite{MR2554935, MR2509322} the $s$-invariants are defined via a filtered complex rather than a graded complex. We can obtain the filtered version as follows. Consider the base change $\eta\colon \F[h] \to \F$ sending $h$ to $1$. This factors through $\F[h,h^{-1}]$ and we get a Frobenius system as in Example \ref{ex:easy_change} that we denote here as $\F_f$. The natural projection
\[
\eta\colon \Csl(D;\Fr_f) \to \Csl(D;\F_f)
\]
becomes injective, if we restrict to a particular grading, that is, to $\Csl^{\ast, j}(D;\Fr_f)$. This leads to the filtration
\[
\cdots F_j \subset F_{j+2} \subset \cdots \subset \Csl(D;\F_f),
\]
for even $j$ with $F_j = \eta(\Csl^{\ast, j}(D;\Fr_f))$. In particular, $F_j/ F_{j-2} \cong \Csl(D;\F)$, the $\mathfrak{sl}_3$-link homology corresponding to the polynomial $X^3$. This gives rise to a spectral sequence $(E^{i,j}_n)$ starting with $E^{i,j}_1 = \Hsl^{i,j}(L;\F)$ converging to $\Hsl(L;\F_f)$.

If $\chr \F \not=3$, we have $E^{i,j}_2 = E^{i,j}_1$ by Proposition \ref{prp:stand_forms}. Furthermore, for $f(X) = X^3-1$ the relevant pages are $E_{1+3k}$, while for $f(X) = X^3 - X$ the relevant pages are $E_{1+2k}$. In the case of $X^3-X$ this is because the cochain complex $\Csl(D;\F_f)$ splits into a direct sum
\[
\Csl(D;\F_{X^3-X}) \cong C_0\oplus C_2
\]
of cochain complexes, where $C_0$ is generated by the generators whose $q$-degree is $0 \bmod 4$, and $C_2$ by those with $q$-degree $2\bmod 4$. We note that this holds in any characteristic.

Given a link diagram $D$ denote the cocycle arising from the canonical coloring of $D$ using the constant function $0$, and giving rise to $I_0\in H^0(L;\F_{X^3-X})$, by $I_0(D) \in \Csl(D;\F_{X^3-X})$.

\begin{lemma}
Let $D$ be a diagram for the link $L$. Then $I_0(D) \in C_{2\cdot c}$, where $c$ is the number of components of $L$ $\bmod\, 2$.
\end{lemma}

\begin{proof}
This is true for the standard diagram $O$ of the unknot, since then $I_0(O) = - (X^2-1) \in \F[X]/(X^3-X)$, and both $X^2$ and $1$ are generators of $q$-degree $2\bmod 4$.

Now if $D$ is an arbitrary link diagram, we can turn it into the standard unknot diagram in a finite sequence of Morse and Reidemeister moves. So, if the statement being true for a diagram $D'$ implies the statement being true for any $D$ obtained from $D'$ by a Morse or a Reidemeister move, we get the result.

If $D'$ is obtained from $D$ by a birth, $I_0(D') = -I_0(D)\otimes (X^2-1)$ and the $q$-degree changes by $2 \bmod 4$, but we also have an extra component. A symmetric argument works if $D'$ is obtained from $D$ by a death.

If $D'$ is obtained from $D$ by an orientation-preserving surgery, again $I_0(D')$ differs from $I_0(D)$ in that one has an extra tensor factor of $-(X^2-1)$. Since the number of components also changes by $1$, the correct statement for $D'$ implies the correct statement for $D$.

If $D'$ is obtained from $D$ by a Reidemeister I move, one of $I_0(D')$ and $I_0(D)$ has an extra factor $-(X^2-1)$ without producing a new component. However, the extra crossing results in a global shif in $q$-degree by $\pm 2$ on the cochain complex, which compensates for this.

If $D'$ is obtained from $D$ by a Reidemeister II move, we have to distinguish two cases, namely whether the two arcs point in the same direction or not. If they point in the same direction, $I_0(D') = I_0(D)$. If they point in opposite directions, we get a small extra circle in one of the smoothings, but the rest also differs by an orientation-preserving surgery. Then either $I_0(D) = I_0(D')$, or they differ by a factor of $(X^2-1)\otimes (X^2-1)$. The two extra crossings have opposite sign, so do not contribute to a global shift in $q$-degree. In particular, the correct statement for $D'$ implies the correct statement for $D$.

If $D'$ is obtained from $D$ by a Reidemeister III move, have again two cases to distinguish. One case leads to identical smoothings, while in the other case the two smoothings differ by two orientation-preserving surgeries. As the crossing signs are the same, we again obtain the correct statement for $D$ from the correct statement for $D'$.
\end{proof}

\begin{proof}[Proof of Theorem \ref{thm:first_main}]
The first two properties follow from the usual arguments. For the third property, note that since $I_0(D) \in C_2$ for a knot diagram, we get $s'''_{X^3-X}(K)\in 4\Z$, and therefore $s^3_{\mathfrak{sl}_3}(K)\in 2\Z$. Since $s^3_{\mathfrak{sl}_3}(T(2,3)) = 2$, the result follows.
\end{proof}

\section{A tangle version}

In order to perform effective computations, a tangle version for the $\mathfrak{sl}_3$-cochain complexes is needed. In \cite{MR2457839} such a tangle version is developed, but only after inverting $3$. Lewark \cite{MR3248745} describes a tangle version that allows for computations over $\Z$, albeit only for $f(X) = X^3$. From these two papers one can readily interpolate how to get the right setting for the correct tangle version over $\ZMV$, and we will only describe certain properties useful for computations.

Given a link diagram $D$, we can always find a sequence of tangle diagrams
\[
\emptyset = T_0 \subset T_1 \subset \cdots \subset T_n = D,
\]
with $T_i - T_{i-1}$ being the diagram of a single crossing. Furthermore, we can assume that each $T_i \subset B_i$ with $B_i$ a closed disc, $B_i = B_{i-1} \cup C_i$, $B_{i-1}\cap C_i = J_i$ an interval, and $C_i$ a closed disc containing $T_i-T_{i-1}$.

Since we are mostly interested in applying the scanning technique for fast calculations developed by Bar-Natan \cite{MR2320156} using the sequence of tangles above, we will restrict ourselves to tangles contained in discs $B$. Given a closed disc $B \subset \R^2$, let $\dot{B}\subset \partial B$ be an oriented $0$-dimensional compact manifold bordant to the empty set. 

A {\em web $\Gamma$ in $(B,\dot{B})$} is a finite oriented graph in $B$ which intersects $\partial B$ transversely in $\dot{B}$, possibly with vertex-less loops, such that its vertices are either of degree $3$ or of degree $1$. At vertices of degree $3$ all edges are either incoming, or all are outgoing. The vertices of degree $1$ agree with $\dot{B}$, and the induced orientation from the edge agrees with the orientation at the point in $\dot{B}$.

Foams between webs in $(B,\dot{B})$ are defined similarly, with the condition that near $\dot{B}$ they are the product foam. Then $\Fm(B,\dot{B})$ is defined as the category whose objects are webs in $(B,\dot{B})$ and whose morphisms are $\ZMV$-linear combinations of isotopy classes of foams.

As before, we want to consider $\Fm(B,\dot{B})$ modulo the relations (\ref{eq:3d}), (\ref{eq:cn}), (\ref{eq:s}) and (\ref{eq:theta}), but we will also need a few relations that are automatically true over $\Fm_\ell$.

First we need the following relations, which revolve dots around a singular arc with either orientation.
\begin{align}
\begin{tikzpicture}[baseline={([yshift=-.6ex]current bounding box.center)}, scale = 0.9, transform shape]
\tripod{0}{0}{1}{0}{0}
\node at (1.1,0.45) {$+$};
\tripod{2}{0}{0}{1}{0}
\node at (3.1,0.45) {$+$};
\tripod{4}{0}{0}{0}{1}
\end{tikzpicture}
& = \fa\hspace{0.2cm}
\begin{tikzpicture}[baseline={([yshift=-.6ex]current bounding box.center)}, scale = 0.9, transform shape]
\tripod{0}{0}{0}{0}{0}
\end{tikzpicture}
\tag{RA}\label{eq:ra}\\
\begin{tikzpicture}[baseline={([yshift=-.6ex]current bounding box.center)}, scale = 0.9, transform shape]
\tripod{0}{0}{1}{1}{0}
\node at (1.1,0.45) {$+$};
\tripod{2}{0}{0}{1}{1}
\node at (3.1,0.45) {$+$};
\tripod{4}{0}{1}{0}{1}
\end{tikzpicture}
& = -\fb \hspace{0.2cm}
\begin{tikzpicture}[baseline={([yshift=-.6ex]current bounding box.center)}, scale = 0.9, transform shape]
\tripod{0}{0}{0}{0}{0}
\end{tikzpicture}
\tag{RB}\label{eq:rb}\\
\begin{tikzpicture}[baseline={([yshift=-.6ex]current bounding box.center)}, scale = 0.9, transform shape]
\tripod{0}{0}{1}{1}{1}
\end{tikzpicture}
& = \fc \hspace{0.2cm}
\begin{tikzpicture}[baseline={([yshift=-.6ex]current bounding box.center)}, scale = 0.9, transform shape]
\tripod{0}{0}{0}{0}{0}
\end{tikzpicture}
\tag{RC}\label{eq:rc}
\end{align}

If the arc is part of a circle, these relations can be shown to hold using (\ref{eq:cn}) and (\ref{eq:theta}), so we really only need them for a singular arc which is not a circle.

We also need the {\em digon removal} relation
\begin{equation}\label{eq:dr}
\begin{tikzpicture}[baseline={([yshift=-.6ex]current bounding box.center)}, scale = 0.825, transform shape]
\shade[color = gray, opacity = 0.5] (0,0) -- (0.5, 0) -- (0.5,1.5) -- (0, 1.5);
\draw (0,0) -- (0.5,0);
\draw (0, 1.5) -- (0.5, 1.5);
\shade[color = gray, opacity = 0.5] (1.2,0) -- (1.7, 0) -- (1.7,1.5) -- (1.2, 1.5);
\draw (1.2,0) -- (1.7, 0);
\draw (1.7,1.5) -- (1.2, 1.5);
\shade[ball color = gray!40, opacity = 0.3] (0.5,0) to [out = 45, in = 180] (0.85, 0.15) to [out = 0, in = 135] (1.2,0) to (1.2, 1.5) to [out = 135, in = 0] (0.85, 1.65) to [out = 180, in = 45] (0.5, 1.5);
\draw (0.5,0) to [out = 45, in = 180] (0.85, 0.15) to [out = 0, in = 135] (1.2,0);
\draw (1.2, 1.5) to [out = 135, in = 0] (0.85, 1.65) to [out = 180, in = 45] (0.5, 1.5);
\shade[ball color = gray!40, opacity = 0.5] (0.5, 0) to [out = 315, in = 180] (0.85, -0.15) to [out = 0, in = 225] (1.2, 0) to (1.2,1.5) to [out = 225, in = 0] (0.85, 1.35) to [out = 180, in = 315] (0.5, 1.5);  
\draw (0.5, 0) to [out = 315, in = 180] (0.85, -0.15) to [out = 0, in = 225] (1.2, 0);
\draw (1.2,1.5) to [out = 225, in = 0] (0.85, 1.35) to [out = 180, in = 315] (0.5, 1.5);
\draw[very thick, ->] (0.5,0) -- (0.5,0.8);
\draw[very thick] (0.5, 0.7) -- (0.5,1.5);
\draw[very thick, <-] (1.2,0.7) -- (1.2,1.5);
\draw[very thick] (1.2, 0) -- (1.2, 0.8);
\end{tikzpicture}
=
\begin{tikzpicture}[baseline={([yshift=-.6ex]current bounding box.center)}, scale = 0.825, transform shape]
\shade[color = gray, opacity = 0.5] (0,0) -- (0.5, 0) to [out = 90, in = 180]  (0.85,0.5) to [out = 0, in = 90] (1.2,0) -- (1.7, 0) -- (1.7, 1.5) -- (1.2, 1.5) to [out = 270, in = 0] (0.85, 1) to [out = 180, in = 270] (0.5, 1.5) -- (0, 1.5);
\draw (0,0) -- (0.5,0);
\draw (0, 1.5) -- (0.5, 1.5);
\draw (1.2,0) -- (1.7, 0);
\draw (1.7,1.5) -- (1.2, 1.5);
\shade[ball color = gray!40, opacity = 0.3] (0.5,0) to [out = 45, in = 180] (0.85, 0.15) to [out = 0, in = 135] (1.2,0) to [out = 90, in = 0] (0.85, 0.5) to [out = 180, in = 90] (0.5, 0);
\shade[ball color = gray!40, opacity = 0.3] (0.5, 1.5) to [out = 45, in = 180] (0.85, 1.65) to [out = 0, in = 135] (1.2, 1.5) to [out = 270, in = 0] (0.85, 1) to [out = 180, in = 270] (0.5, 1.5);
\draw (0.5,0) to [out = 45, in = 180] (0.85, 0.15) to [out = 0, in = 135] (1.2,0);
\draw (1.2, 1.5) to [out = 135, in = 0] (0.85, 1.65) to [out = 180, in = 45] (0.5, 1.5);
\shade[ball color = gray!40, opacity = 0.5] (0.5, 0) to [out = 315, in = 180] (0.85, -0.15) to [out = 0, in = 225] (1.2, 0) to [out = 90, in = 0] (0.85, 0.5) to [out = 180, in = 90] (0.5, 0); 
\shade[ball color = gray!40, opacity = 0.5] (0.5, 1.5) to [out = 315, in = 180] (0.85, 1.35) to [out = 0, in = 225] (1.2, 1.5) to [out = 270, in = 0] (0.85, 1) to [out = 180, in = 270] (0.5, 1.5);
\draw (0.5, 0) to [out = 315, in = 180] (0.85, -0.15) to [out = 0, in = 225] (1.2, 0);
\draw (1.2,1.5) to [out = 225, in = 0] (0.85, 1.35) to [out = 180, in = 315] (0.5, 1.5);
\draw[very thick, ->] (0.5,0) to [out = 90, in = 180]  (0.85,0.5);
\draw[very thick] (0.85, 0.5) to [out = 0, in = 90] (1.2,0);
\draw[very thick, <-] (0.85,1) to [out = 0, in = 270] (1.2,1.5);
\draw[very thick] (0.5, 1.5) to [out = 270, in = 180] (0.85, 1);
\node at (0.85, 1.5) {$\bullet$};
\end{tikzpicture}
-
\begin{tikzpicture}[baseline={([yshift=-.6ex]current bounding box.center)}, scale = 0.825, transform shape]
\shade[color = gray, opacity = 0.5] (0,0) -- (0.5, 0) to [out = 90, in = 180]  (0.85,0.5) to [out = 0, in = 90] (1.2,0) -- (1.7, 0) -- (1.7, 1.5) -- (1.2, 1.5) to [out = 270, in = 0] (0.85, 1) to [out = 180, in = 270] (0.5, 1.5) -- (0, 1.5);
\draw (0,0) -- (0.5,0);
\draw (0, 1.5) -- (0.5, 1.5);
\draw (1.2,0) -- (1.7, 0);
\draw (1.7,1.5) -- (1.2, 1.5);
\shade[ball color = gray!40, opacity = 0.3] (0.5,0) to [out = 45, in = 180] (0.85, 0.15) to [out = 0, in = 135] (1.2,0) to [out = 90, in = 0] (0.85, 0.5) to [out = 180, in = 90] (0.5, 0);
\shade[ball color = gray!40, opacity = 0.3] (0.5, 1.5) to [out = 45, in = 180] (0.85, 1.65) to [out = 0, in = 135] (1.2, 1.5) to [out = 270, in = 0] (0.85, 1) to [out = 180, in = 270] (0.5, 1.5);
\draw (0.5,0) to [out = 45, in = 180] (0.85, 0.15) to [out = 0, in = 135] (1.2,0);
\draw (1.2, 1.5) to [out = 135, in = 0] (0.85, 1.65) to [out = 180, in = 45] (0.5, 1.5);
\shade[ball color = gray!40, opacity = 0.5] (0.5, 0) to [out = 315, in = 180] (0.85, -0.15) to [out = 0, in = 225] (1.2, 0) to [out = 90, in = 0] (0.85, 0.5) to [out = 180, in = 90] (0.5, 0); 
\shade[ball color = gray!40, opacity = 0.5] (0.5, 1.5) to [out = 315, in = 180] (0.85, 1.35) to [out = 0, in = 225] (1.2, 1.5) to [out = 270, in = 0] (0.85, 1) to [out = 180, in = 270] (0.5, 1.5);
\draw (0.5, 0) to [out = 315, in = 180] (0.85, -0.15) to [out = 0, in = 225] (1.2, 0);
\draw (1.2,1.5) to [out = 225, in = 0] (0.85, 1.35) to [out = 180, in = 315] (0.5, 1.5);
\draw[very thick, ->] (0.5,0) to [out = 90, in = 180]  (0.85,0.5);
\draw[very thick] (0.85, 0.5) to [out = 0, in = 90] (1.2,0);
\draw[very thick, <-] (0.85,1) to [out = 0, in = 270] (1.2,1.5);
\draw[very thick] (0.5, 1.5) to [out = 270, in = 180] (0.85, 1);
\node at (0.85, 0) {$\bullet$};
\end{tikzpicture}
\tag{DR}
\end{equation}

and the {\em square removal} relation, compare \cite[Lm.2.3]{MR2336253}.

\begin{equation}
\label{eq:sr}
\begin{tikzpicture}[baseline={([yshift=-.6ex]current bounding box.center)}, scale = 0.825, transform shape]
\shade[color = gray, opacity = 0.3] (0.4, 0.7) -- (0.9, 0.3) -- (0.9, 1.8) -- (0.4, 2.2);
\draw (0.5,0) -- (0.9, 0.3) -- (0.4, 0.7);
\shade[color = gray, opacity = 0.4] (0.5, 0) -- (0.9, 0.3) -- (0.9, 1.8) -- (0.5, 1.5);
\shade[color = gray, opacity = 0.6] (0, -0.1) -- (0.5, 0) -- (0.5, 1.5) -- (0, 1.4);
\shade[color = gray, opacity = 0.3] (0.9, 0.3) -- (1.6, 0.3) -- (1.6, 1.8) -- (0.9, 1.8);
\draw (0.9, 0.3) -- (1.6, 0.3);
\draw (0.9, 1.8) -- (1.6, 1.8);
\shade[color = gray, opacity = 0.3] (1.6, 0.3) -- (2.1, 0.4) -- (2.1, 1.9) -- (1.6, 1.8);
\draw (1.6, 0.3) -- (2.1, 0.4);
\draw (1.6, 1.8) -- (2.1, 1.9);
\draw [very thick, ->] (0.9, 0.3) -- (0.9, 1.1);
\draw [very thick] (0.9, 1) -- (0.9, 1.8); 
\draw [very thick, <-] (1.6, 1) -- (1.6, 1.8);
\draw [very thick] (1.6, 0.3) -- (1.6, 1.1);
\shade[color = gray, opacity = 0.4] (1.2, 0) -- (1.6, 0.3) -- (1.6, 1.8) -- (1.2, 1.5);
\draw (1.2, 0) -- (1.6, 0.3);
\draw (1.2, 1.5) -- (1.6, 1.8);
\shade[color = gray, opacity = 0.5] (0.5, 0) -- (1.2, 0) -- (1.2, 1.5) -- (0.5, 1.5);
\draw (0.5, 0) -- (1.2, 0);
\draw (0.5, 1.5) -- (1.2, 1.5);
\shade[color = gray, opacity = 0.6] (1.2, 0) -- (1.7, -0.4) -- (1.7, 1.1) -- (1.2, 1.5);
\draw (1.2, 0) -- (1.7, -0.4);
\draw (1.2, 1.5) -- (1.7, 1.1);
\draw [very thick] (0.5, 0) -- (0.5, 0.8);
\draw [very thick, <-] (0.5, 0.7) -- (0.5, 1.5);
\draw [very thick, ->] (1.2, 0) -- (1.2, 0.8);
\draw [very thick] (1.2, 0.7) -- (1.2, 1.5);
\draw (0, -0.1) -- (0.5,0);
\draw (0, 1.4) -- (0.5, 1.5) -- (0.9, 1.8) -- (0.4, 2.2); 
\end{tikzpicture}
= - \,
\begin{tikzpicture}[baseline={([yshift=-.6ex]current bounding box.center)}, scale = 0.825, transform shape]
\shade[color = gray, opacity = 0.3] (0.4, 0.7) -- (0.9, 0.3) to [out = 90, in = 180] (1.25, 0.5) to [out = 0, in = 90] (1.6, 0.3) -- (2.1, 0.4) -- (2.1, 1.9) -- (1.6, 1.8) to [out = 270, in = 0] (1.25, 1.6) to [out = 180, in = 270] (0.9, 1.8) -- (0.4, 2.2);
\draw (0.4, 0.7) -- (0.9, 0.3);
\draw (1.6, 0.3) -- (2.1, 0.4);
\draw (0.4, 2.2) -- (0.9, 1.8);
\draw (1.6, 1.8) -- (2.1, 1.9);
\shade[color = gray, opacity = 0.3] (0.9, 1.8) to [out = 270, in = 180] (1.25, 1.6) to [out = 0, in =270] (1.6, 1.8);
\draw (0.9, 1.8) -- (1.6, 1.8);
\shade[color = gray, opacity = 0.9] (0.9, 1.8) -- (0.5, 1.5) to [out = 270, in = 180] (0.85, 1.3) to [out = 0, in = 270] (1.2, 1.5) -- (1.6, 1.8) to [out = 270, in = 0] (1.25, 1.6) to [out = 180, in = 270] (0.9, 1.8);
\draw [very thick, ->] (1.6, 1.8) to [out = 270, in = 0] (1.25, 1.6);
\draw [very thick] (1.25, 1.6) to [out = 180, in = 270] (0.9, 1.8);
\draw (0.9, 1.8) -- (0.5, 1.5);
\draw (1.6, 1.8) -- (1.2, 1.5);
\shade[color = gray, opacity = 0.5] (0.5, 1.5) to [out = 270, in = 180] (0.85, 1.3) to [out = 0, in = 270] (1.2, 1.5);
\draw (0.5, 1.5) -- (1.2, 1.5);
\shade[color = gray, opacity = 0.3] (0.9, 0.3) -- (1.6, 0.3) to [out = 90, in = 0] (1.25, 0.5) to [out = 180, in = 90] (0.9, 0.3);
\draw (0.9, 0.3) -- (1.6, 0.3);
\shade[color = gray, opacity = 0.7] (1.2, 0) -- (1.6, 0.3) to [out = 90, in = 0] (1.25, 0.5) to [out = 180, in = 90] (0.9, 0.3) -- (0.5, 0) to [out = 90, in = 180] (0.85, 0.2) to [out = 0, in = 90] (1.2, 0);
\draw [very thick, ->] (0.9, 0.3) to [out = 90, in = 180] (1.25, 0.5);
\draw [very thick] (1.25, 0.5) to [out = 0, in = 90] (1.6, 0.3); 
\draw (0.5, 0) -- (0.9, 0.3);
\draw (1.2, 0) -- (1.6, 0.3);
\shade[color = gray, opacity = 0.6] (0, -0.1) -- (0.5, 0) to [out = 90, in = 180] (0.85, 0.2) to [out = 0, in = 90] (1.2, 0) -- (1.7, -0.4) -- (1.7, 1.1) -- (1.2, 1.5) to [out = 270, in = 0] (0.85, 1.3) to [out = 180, in = 270] (0.5, 1.5) -- (0, 1.4);
\shade[color = gray, opacity = 0.5] (0.5, 0) to [out = 90, in = 180] (0.85, 0.2) to [out = 0, in = 90] (1.2, 0);
\draw (0, -0.1) -- (0.5, 0) -- (1.2, 0) -- (1.7, -0.4);
\draw (0, 1.4) -- (0.5, 1.5);
\draw (1.2, 1.5) -- (1.7, 1.1);
\draw[very thick, ->] (1.2, 0) to [out = 90, in = 0] (0.85, 0.2);
\draw[very thick] (0.85, 0.2) to [out = 180, in = 90] (0.5, 0);
\draw [very thick, ->] (0.5, 1.5) to [out = 270, in = 180] (0.85, 1.3);
\draw [very thick] (0.85, 1.3) to [out = 0, in = 270] (1.2, 1.5);
\end{tikzpicture}
\, - \,
\begin{tikzpicture}[baseline={([yshift=-.6ex]current bounding box.center)}, scale = 0.825, transform shape]
\shade[color = gray, opacity = 0.3] (0.4, 0.7) -- (0.9, 0.3) to [out = 110, in = 250] (0.9, 1.8) -- (0.4, 2.2);
\shade[color = gray, opacity = 0.3] (2.1, 1.9) -- (1.6, 1.8) to [out = 260, in = 0] (1.3, 1.4) to [out = 180, in = 280] (1.2, 1.5) to [out = 290, in = 70] (1.2, 0) to [out = 80, in = 100] (1.6, 0.3) -- (2.1, 0.4); 
\draw (0.4, 0.7) -- (0.9, 0.3);
\draw (1.6, 0.3) -- (2.1, 0.4);
\draw (0.4, 2.2) -- (0.9, 1.8);
\draw (1.6, 1.8) -- (2.1, 1.9);
\draw (0.5, 0) -- (0.9, 0.3) -- (1.6, 0.3) -- (1.2, 0);
\shade[color = gray, opacity = 0.5] (0.8, 1.5) -- (1.2, 1.5) to [out = 280, in = 180] (1.3, 1.4) to [out = 0, in = 260] (1.6, 1.8) -- (0.9, 1.8) to [out = 260, in = 45] (0.8, 1.5);
\shade[color = gray, opacity = 0.6] (1.2, 1.5) to [out = 280, in = 180] (1.3, 1.4) to [out = 0, in = 260] (1.6, 1.8);
\shade[color = gray, opacity = 0.6] (0.5, 1.5) to [out = 280, in = 180] (0.6, 1.4) to [out = 0, in = 260] (0.9, 1.8);
\shade[color = gray, opacity = 0.7] (0.5, 0) -- (0.9, 0.3) to [out = 100, in = 0] (0.8, 0.4) to [out = 180, in = 45] (0.6, 0.3)to [out = 225, in = 80] (0.5, 0);
\shade[color = gray, opacity = 0.7] (1.2, 0) -- (1.6, 0.3) to [out = 100, in = 0] (1.5, 0.4) to [out = 180, in = 45] (1.3, 0.3)to [out = 225, in = 80] (1.2, 0);
\shade[color = gray, opacity = 0.6] (0.5, 0) to [out = 80, in = 180] (0.8, 0.4) -- (1.5, 0.4) to [out = 180, in = 80] (1.2, 0);
\shade[color = gray, opacity = 0.6] (0, -0.1) -- (0.5, 0) to [out = 80, in = 225] (0.6, 0.3) to [out = 45, in = 180] (0.8, 0.4) to [out = 100, in = 260] (0.8, 1.5) to [out = 225, in = 0] (0.6, 1.4) to [out = 180, in = 280] (0.5, 1.5) -- (0, 1.4);
\shade[color = gray, opacity = 0.6] (1.7, -0.4) -- (1.2, 0) to [out = 80, in = 280] (1.2, 1.5) -- (1.7, 1.1);
\draw [very thick, ->] (1.6, 1.8) to [out = 260, in = 0] (1.3, 1.4);
\draw [very thick] (1.3, 1.4) to [out = 180, in = 280] (1.2, 1.5);
\draw [very thick] (0.6, 1.4) to [out = 0, in = 260] (0.9, 1.8);
\draw [very thick, ->] (0.5, 1.5) to [out = 280, in = 180] (0.6, 1.4) to  [out = 0, in = 225] (0.8, 1.5);
\shade[color = gray, opacity = 0.6] (0.5,1.5) -- (1.2, 1.5) to [out = 280, in = 180] (1.3, 1.4) -- (0.6, 1.4) to [out = 180, in = 280] (0.5, 1.5);
\draw [very thick, ->] (1.2, 0) to [out = 80, in = 180] (1.5, 0.4);
\draw [very thick] (1.5, 0.4) to [out = 0, in = 100] (1.6, 0.3);
\draw [very thick, ->] (0.9, 0.3) to [out = 100, in = 0] (0.8, 0.4) to [out = 180, in = 45] (0.6, 0.3);
\draw [very thick] (0.6, 0.3) to [out = 225, in = 80] (0.5, 0);
\draw (0, -0.1) -- (0.5, 0) -- (1.2, 0) -- (1.7, -0.4);
\draw (0, 1.4) -- (0.5, 1.5) -- (1.2, 1.5) -- (1.7, 1.1);
\draw (0.5, 1.5) -- (0.9, 1.8) -- (1.6, 1.8) -- (1.2, 1.5);
\end{tikzpicture}
\tag{SR}
\end{equation}

Define $\Fm_\ell(B,\dot{B})$ to be the quotient of $\Fm(B,\dot{B})$ modulo the relations  (\ref{eq:3d}), (\ref{eq:cn}), (\ref{eq:s}), (\ref{eq:theta}), (\ref{eq:ra}), (\ref{eq:rb}), (\ref{eq:rc}), (\ref{eq:dr}), and (\ref{eq:sr}).

We can now define the cochain complex $\Csl(T;\Fr_{MV})$ of a tangle diagram $T\subset B$ over $\Fm_\ell(B,\dot{B})$ as before, and get it to be invariant under Reidemeister moves as in \cite[\S 2.2]{MR2336253}. Furthermore, we can remove digons and squares inside webs as in \cite[($D^c$), ($S^c$)]{MR3248745}. We can also remove loops, but the isomorphism is slightly more complicated than in \cite{MR3248745} because of the more complicated (\ref{eq:cn}) relation. It is given by
\[
\begin{tikzpicture}
\matrix [matrix of math nodes,left delimiter={(},right delimiter={)}] at (0,1.5)
{
- \, 
\begin{tikzpicture}[baseline={([yshift=-.6ex]current bounding box.center)}]
\bowlud{0}{0}{0.75}{0.05}
\node[scale = 0.75] at (0.3,0) {$\bullet\,\bullet$};
\end{tikzpicture} 
+ \fa \, 
\begin{tikzpicture}[baseline={([yshift=-.6ex]current bounding box.center)}]
\bowlud{0}{0}{0.75}{0.05}
\node[scale = 0.75] at (0.3,0) {$\bullet$};
\end{tikzpicture} 
+ \fb \,
\begin{tikzpicture}[baseline={([yshift=-.6ex]current bounding box.center)}]
\bowlud{0}{0}{0.75}{0.05}
\end{tikzpicture} 
\\
- \, 
\begin{tikzpicture}[baseline={([yshift=-.6ex]current bounding box.center)}]
\bowlud{0}{0}{0.75}{0.05}
\node[scale = 0.75] at (0.3,0) {$\bullet$};
\end{tikzpicture} 
+ \fa \, 
\begin{tikzpicture}[baseline={([yshift=-.6ex]current bounding box.center)}]
\bowlud{0}{0}{0.75}{0.05}
\end{tikzpicture} 
\\
-\,
\begin{tikzpicture}[baseline={([yshift=-.6ex]current bounding box.center)}]
\bowlud{0}{0}{0.75}{0.05}
\end{tikzpicture} 
\\
};
\draw[->] (-2.5, 0) -- (2.5, 0);
\node at (3,0) {$q^{-2}\,\emptyset$};
\node at (3, 0.5) {$\oplus$};
\node at (3,1) {$\emptyset$};
\node at (3,1.5) {$\oplus$};
\node at (3,2) {$q^2\,\emptyset$};
\matrix [matrix of math nodes,left delimiter={(},right delimiter={)}] at (5.2,0.85) {
\begin{tikzpicture}[baseline={([yshift=-.6ex]current bounding box.center)}]
\bowl{0}{0}{0.75}{0.05}
\end{tikzpicture} 
&
\begin{tikzpicture}[baseline={([yshift=-.6ex]current bounding box.center)}]
\bowl{0}{0}{0.75}{0.05}
\node[scale = 0.75] at (0.3,-0.1) {$\bullet$};
\end{tikzpicture} 
&
\begin{tikzpicture}[baseline={([yshift=-.6ex]current bounding box.center)}]
\bowl{0}{0}{0.75}{0.05}
\node[scale = 0.75] at (0.3,-0.1) {$\bullet\,\bullet$};
\end{tikzpicture} \\
};
\draw [->] (3.5, 0) -- (7, 0);
\draw[thick] (-3,0) circle (0.25);
\draw[thick] (7.5,0) circle (0.25);
\end{tikzpicture}
\]

To run the usual scanning algorithm of \cite{MR2320156} effectively, the only thing missing is a good way to simplify foams. In particular, we need the singular circle removal relation described in \cite[Fig.5]{MR3248745}. To see that we get the same relation, let us first derive the {\em bursting bubbles} relation of \cite[Fig.18]{MR2100691}.

\begin{lemma}\label{lm:burst_bub}
The following relations hold in $\Fm_\ell(B,\dot{B})$.
\begin{align*}
\begin{tikzpicture}[baseline={([yshift=-.6ex]current bounding box.center)}]
\bubble
\end{tikzpicture}
\, &= \, 0\,  = \,
\begin{tikzpicture}[baseline={([yshift=-.6ex]current bounding box.center)}]
\bubble
\node[scale = 0.75] at (0.4, 0.4) {$\bullet$};
\node[scale = 0.75] at (0.4, -0.4) {$\bullet$};
\end{tikzpicture}
\\
\begin{tikzpicture}[baseline={([yshift=-.6ex]current bounding box.center)}]
\bubble
\node[scale = 0.75] at (0.4, -0.4) {$\bullet$};
\end{tikzpicture}
\, & = \,
\begin{tikzpicture}[baseline={([yshift=-.6ex]current bounding box.center)}]
\fill[color = gray, opacity = 0.4] (-0.65,-0.325) -- (1, -0.325) -- (1.5, 0.325) -- (-0.15, 0.325); 
\draw (-0.65,-0.325) -- (1, -0.325) -- (1.5, 0.325) -- (-0.15, 0.325) -- (-0.65,-0.325); 
\end{tikzpicture}
\\
\begin{tikzpicture}[baseline={([yshift=-.6ex]current bounding box.center)}]
\bubble
\node[scale = 0.75] at (0.4, -0.4) {$\bullet\, \bullet$};
\end{tikzpicture}
\, &=\, - \,
\begin{tikzpicture}[baseline={([yshift=-.6ex]current bounding box.center)}]
\fill[color = gray, opacity = 0.4] (-0.65,-0.325) -- (1, -0.325) -- (1.5, 0.325) -- (-0.15, 0.325); 
\draw (-0.65,-0.325) -- (1, -0.325) -- (1.5, 0.325) -- (-0.15, 0.325) -- (-0.65,-0.325); 
\node[scale = 0.75] at (0.4, 0) {$\bullet$};
\end{tikzpicture}
\, + \fa \, 
\begin{tikzpicture}[baseline={([yshift=-.6ex]current bounding box.center)}]
\fill[color = gray, opacity = 0.4] (-0.65,-0.325) -- (1, -0.325) -- (1.5, 0.325) -- (-0.15, 0.325); 
\draw (-0.65,-0.325) -- (1, -0.325) -- (1.5, 0.325) -- (-0.15, 0.325) -- (-0.65,-0.325); 
\end{tikzpicture}
\end{align*}
\end{lemma}

\begin{proof}
We have a cylinder just outside the singular circle. Performing the surgery leads to the plane and a $\Theta$-foam. The result follows from (\ref{eq:cn}) and (\ref{eq:theta}). Note that \cite[Fig.18]{MR2100691} uses a different orientation on the singular circle, hence the difference in signs.
\end{proof}

\begin{lemma}
The following relation holds in $\Fm_\ell(B,\dot{B})$.
\begin{align*}
\begin{tikzpicture}[baseline={([yshift=-.6ex]current bounding box.center)}]
\cylinderslice
\end{tikzpicture}
\, & = \,
\begin{tikzpicture}[baseline={([yshift=-.6ex]current bounding box.center)}]
\resolveslice
\node[scale = 0.75] at (0.4, 0.8) {$\bullet\,\bullet$};
\node[scale = 0.75] at (0.4, -0.8) {$\bullet$};
\end{tikzpicture}
\, + \,
\begin{tikzpicture}[baseline={([yshift=-.6ex]current bounding box.center)}]
\resolveslice
\node[scale = 0.75] at (0.4, 0) {$\bullet\,\bullet$};
\node[scale = 0.75] at (0.4, 0.8) {$\bullet$};
\end{tikzpicture}
\, + \,
\begin{tikzpicture}[baseline={([yshift=-.6ex]current bounding box.center)}]
\resolveslice
\node[scale = 0.75] at (0.4, -0.8) {$\bullet\,\bullet$};
\node[scale = 0.75] at (0.4, 0) {$\bullet$};
\end{tikzpicture}
\\
& \, - \,
\begin{tikzpicture}[baseline={([yshift=-.6ex]current bounding box.center)}]
\resolveslice
\node[scale = 0.75] at (0.4, 0.8) {$\bullet\,\bullet$};
\node[scale = 0.75] at (0.4, 0) {$\bullet$};
\end{tikzpicture}
\, - \,
\begin{tikzpicture}[baseline={([yshift=-.6ex]current bounding box.center)}]
\resolveslice
\node[scale = 0.75] at (0.4, -0.8) {$\bullet\,\bullet$};
\node[scale = 0.75] at (0.4, 0.8) {$\bullet$};
\end{tikzpicture}
\, - \,
\begin{tikzpicture}[baseline={([yshift=-.6ex]current bounding box.center)}]
\resolveslice
\node[scale = 0.75] at (0.4, 0) {$\bullet\,\bullet$};
\node[scale = 0.75] at (0.4, -0.8) {$\bullet$};
\end{tikzpicture}
\end{align*}
\end{lemma}

\begin{proof}
By (\ref{eq:cn}) we have
\begin{align*}
\begin{tikzpicture}[baseline={([yshift=-.6ex]current bounding box.center)}, scale = 0.55, transform shape]
\cylinderslice
\end{tikzpicture}
\, & = \, - \,
\begin{tikzpicture}[baseline={([yshift=-.6ex]current bounding box.center)}, scale = 0.55, transform shape]
\resolvetop
\node[scale = 0.75] at (0.4, 1) {$\bullet\,\bullet$}; 
\end{tikzpicture}
\, - \,
\begin{tikzpicture}[baseline={([yshift=-.6ex]current bounding box.center)}, scale = 0.55, transform shape]
\resolvetop
\node[scale = 0.75] at (0.4, 1) {$\bullet$}; 
\node[scale = 0.75] at (0.4, 0.1) {$\bullet$}; 
\end{tikzpicture}
\, - \,
\begin{tikzpicture}[baseline={([yshift=-.6ex]current bounding box.center)}, scale = 0.55, transform shape]
\resolvetop
\node[scale = 0.75] at (0.4, 0.1) {$\bullet\,\bullet$}; 
\end{tikzpicture}
\, + \fa \,
\begin{tikzpicture}[baseline={([yshift=-.6ex]current bounding box.center)}, scale = 0.55, transform shape]
\resolvetop
\node[scale = 0.75] at (0.4, 1) {$\bullet$}; 
\end{tikzpicture}
\, + \fa \,
\begin{tikzpicture}[baseline={([yshift=-.6ex]current bounding box.center)}, scale = 0.55, transform shape]
\resolvetop
\node[scale = 0.75] at (0.4, 0.1) {$\bullet$}; 
\end{tikzpicture}
\, + \fb \,
\begin{tikzpicture}[baseline={([yshift=-.6ex]current bounding box.center)}, scale = 0.55, transform shape]
\resolvetop
\end{tikzpicture}
\\
& \, = \, - \,
\begin{tikzpicture}[baseline={([yshift=-.6ex]current bounding box.center)}, scale = 0.55, transform shape]
\resolvetop
\node[scale = 0.75] at (0.4, 1) {$\bullet\,\bullet$}; 
\end{tikzpicture}
\, + \,
\begin{tikzpicture}[baseline={([yshift=-.6ex]current bounding box.center)}, scale = 0.55, transform shape]
\resolvetop
\node[scale = 0.75] at (0.4, 1) {$\bullet$}; 
\node[scale = 0.75] at (-0.2, -0.1) {$\bullet$};
\end{tikzpicture}
\, + \,
\begin{tikzpicture}[baseline={([yshift=-.6ex]current bounding box.center)}, scale = 0.55, transform shape]
\resolvetop
\node[scale = 0.75] at (0.4, 1) {$\bullet$}; 
\node[scale = 0.75] at (0.4, -0.6) {$\bullet$};
\end{tikzpicture}
\, + \,
\begin{tikzpicture}[baseline={([yshift=-.6ex]current bounding box.center)}, scale = 0.55, transform shape]
\resolvetop
\node[scale = 0.75] at (0.4, 0.1) {$\bullet$}; 
\node[scale = 0.75] at (-0.2, -0.1) {$\bullet$};
\end{tikzpicture}
\, + \,
\begin{tikzpicture}[baseline={([yshift=-.6ex]current bounding box.center)}, scale = 0.55, transform shape]
\resolvetop
\node[scale = 0.75] at (0.4, 0.1) {$\bullet$}; 
\node[scale = 0.75] at (0.4, -0.6) {$\bullet$};
\end{tikzpicture}
\, + \fb \,
\begin{tikzpicture}[baseline={([yshift=-.6ex]current bounding box.center)}, scale = 0.55, transform shape]
\resolvetop
\end{tikzpicture}
\\
& = \, - \,
\begin{tikzpicture}[baseline={([yshift=-.6ex]current bounding box.center)}, scale = 0.55, transform shape]
\resolvetop
\node[scale = 0.75] at (0.4, 1) {$\bullet\,\bullet$}; 
\end{tikzpicture}
\, + \,
\begin{tikzpicture}[baseline={([yshift=-.6ex]current bounding box.center)}, scale = 0.55, transform shape]
\resolvetop
\node[scale = 0.75] at (0.4, 1) {$\bullet$}; 
\node[scale = 0.75] at (-0.2, -0.1) {$\bullet$};
\end{tikzpicture}
\, + \,
\begin{tikzpicture}[baseline={([yshift=-.6ex]current bounding box.center)}, scale = 0.55, transform shape]
\resolvetop
\node[scale = 0.75] at (0.4, 1) {$\bullet$}; 
\node[scale = 0.75] at (0.4, -0.6) {$\bullet$};
\end{tikzpicture}
\, - \,
\begin{tikzpicture}[baseline={([yshift=-.6ex]current bounding box.center)}, scale = 0.55, transform shape]
\resolvetop
\node[scale = 0.75] at (-0.2, -0.1) {$\bullet$}; 
\node[scale = 0.75] at (0.4, -0.6) {$\bullet$};
\end{tikzpicture}
\end{align*}
where for the second equation we used (\ref{eq:ra}) twice, and for the third (\ref{eq:rb}). Using (\ref{eq:cn}) below the singular circle gives
\begin{align*}
\begin{tikzpicture}[baseline={([yshift=-.6ex]current bounding box.center)}, scale = 0.55, transform shape]
\cylinderslice
\end{tikzpicture}
\, & = \,
\begin{tikzpicture}[baseline={([yshift=-.6ex]current bounding box.center)}, scale = 0.55, transform shape]
\resolvebot
\node[scale = 0.75] at (0.4, 1) {$\bullet\,\bullet$}; 
\node[scale = 0.75] at (0.4, -0.4) {$\bullet\, \bullet$};
\end{tikzpicture}
\, + \,
\begin{tikzpicture}[baseline={([yshift=-.6ex]current bounding box.center)}, scale = 0.55, transform shape]
\resolvebot
\node[scale = 0.75] at (0.4, 1) {$\bullet\,\bullet$}; 
\node[scale = 0.75] at (0.4, -0.4) {$\bullet$}; 
\node[scale = 0.75] at (0.4, -1) {$\bullet$};
\end{tikzpicture}
\, - \fa \,
\begin{tikzpicture}[baseline={([yshift=-.6ex]current bounding box.center)}, scale = 0.55, transform shape]
\resolvebot
\node[scale = 0.75] at (0.4, 1) {$\bullet\,\bullet$}; 
\node[scale = 0.75] at (0.4, -0.4) {$\bullet$}; 
\end{tikzpicture}
\, -  \,
\begin{tikzpicture}[baseline={([yshift=-.6ex]current bounding box.center)}, scale = 0.55, transform shape]
\resolvebot
\node[scale = 0.75] at (0.4, 1) {$\bullet$}; 
\node[scale = 0.75] at (-0.2, -0.1) {$\bullet$};
\node[scale = 0.75] at (0.4, -0.4) {$\bullet\, \bullet$};
\end{tikzpicture}
\, -  \,
\begin{tikzpicture}[baseline={([yshift=-.6ex]current bounding box.center)}, scale = 0.55, transform shape]
\resolvebot
\node[scale = 0.75] at (0.4, 1) {$\bullet$}; 
\node[scale = 0.75] at (-0.2, -0.1) {$\bullet$};
\node[scale = 0.75] at (0.4, -0.4) {$\bullet$};
\node[scale = 0.75] at (0.4, -1) {$\bullet$};
\end{tikzpicture}
\, + \fa  \,
\begin{tikzpicture}[baseline={([yshift=-.6ex]current bounding box.center)}, scale = 0.55, transform shape]
\resolvebot
\node[scale = 0.75] at (0.4, 1) {$\bullet$}; 
\node[scale = 0.75] at (-0.2, -0.1) {$\bullet$};
\node[scale = 0.75] at (0.4, -0.4) {$\bullet$};
\end{tikzpicture}
\\
& \, -  \,
\begin{tikzpicture}[baseline={([yshift=-.6ex]current bounding box.center)}, scale = 0.55, transform shape]
\resolvebot
\node[scale = 0.75] at (0.4, 1) {$\bullet$}; 
\node[scale = 0.75] at (0.4, -1) {$\bullet$};
\node[scale = 0.75] at (0.4, -0.4) {$\bullet\, \bullet$};
\end{tikzpicture}
\, -  \,
\begin{tikzpicture}[baseline={([yshift=-.6ex]current bounding box.center)}, scale = 0.55, transform shape]
\resolvebot
\node[scale = 0.75] at (0.4, 1) {$\bullet$}; 
\node[scale = 0.75] at  (0.4, -1) {$\bullet\,\bullet$};
\node[scale = 0.75] at (0.4, -0.4) {$\bullet$};
\end{tikzpicture}
\, + \fa  \,
\begin{tikzpicture}[baseline={([yshift=-.6ex]current bounding box.center)}, scale = 0.55, transform shape]
\resolvebot
\node[scale = 0.75] at (0.4, 1) {$\bullet$}; 
\node[scale = 0.75] at  (0.4, -1) {$\bullet$};
\node[scale = 0.75] at (0.4, -0.4) {$\bullet$};
\end{tikzpicture}
\, +  \,
\begin{tikzpicture}[baseline={([yshift=-.6ex]current bounding box.center)}, scale = 0.55, transform shape]
\resolvebot
\node[scale = 0.75] at (-0.2, -0.1) {$\bullet$}; 
\node[scale = 0.75] at (0.4, -1) {$\bullet$};
\node[scale = 0.75] at (0.4, -0.4) {$\bullet\, \bullet$};
\end{tikzpicture}
\, +  \,
\begin{tikzpicture}[baseline={([yshift=-.6ex]current bounding box.center)}, scale = 0.55, transform shape]
\resolvebot
\node[scale = 0.75] at (-0.2, -0.1) {$\bullet$}; 
\node[scale = 0.75] at  (0.4, -1) {$\bullet\,\bullet$};
\node[scale = 0.75] at (0.4, -0.4) {$\bullet$};
\end{tikzpicture}
\, - \fa  \,
\begin{tikzpicture}[baseline={([yshift=-.6ex]current bounding box.center)}, scale = 0.55, transform shape]
\resolvebot
\node[scale = 0.75] at (-0.2, -0.1) {$\bullet$}; 
\node[scale = 0.75] at  (0.4, -1) {$\bullet$};
\node[scale = 0.75] at (0.4, -0.4) {$\bullet$};
\end{tikzpicture}
\end{align*}
where we already use Lemma \ref{lm:burst_bub} on the bubbles having no dots. Using the other relations of Lemma \ref{lm:burst_bub} gives
\begin{align*}
\begin{tikzpicture}[baseline={([yshift=-.6ex]current bounding box.center)}, scale = 0.55, transform shape]
\cylinderslice
\end{tikzpicture}
\, & = \, - \,
\begin{tikzpicture}[baseline={([yshift=-.6ex]current bounding box.center)}, scale = 0.55, transform shape]
\resolveslice
\node[scale = 0.75] at (0.4, 0.8) {$\bullet\,\bullet$};
\node[scale = 0.75] at (0.4, 0) {$\bullet$};
\end{tikzpicture}
\, + \,
\begin{tikzpicture}[baseline={([yshift=-.6ex]current bounding box.center)}, scale = 0.55, transform shape]
\resolveslice
\node[scale = 0.75] at (0.4, 0.8) {$\bullet\,\bullet$};
\node[scale = 0.75] at (0.4, -0.8) {$\bullet$};
\end{tikzpicture}
\, + \,
\begin{tikzpicture}[baseline={([yshift=-.6ex]current bounding box.center)}, scale = 0.55, transform shape]
\resolveslice
\node[scale = 0.75] at (0.4, 0) {$\bullet\,\bullet$};
\node[scale = 0.75] at (0.4, 0.8) {$\bullet$};
\end{tikzpicture}
\, - \,
\begin{tikzpicture}[baseline={([yshift=-.6ex]current bounding box.center)}, scale = 0.55, transform shape]
\resolveslice
\node[scale = 0.75] at (0.4, 0) {$\bullet$};
\node[scale = 0.75] at (0.4, 0.8) {$\bullet$};
\node[scale = 0.75] at (0.4, -0.8) {$\bullet$};
\end{tikzpicture}
\\ 
& \, + \,
\begin{tikzpicture}[baseline={([yshift=-.6ex]current bounding box.center)}, scale = 0.55, transform shape]
\resolveslice
\node[scale = 0.75] at (0.4, 0) {$\bullet$};
\node[scale = 0.75] at (0.4, 0.8) {$\bullet$};
\node[scale = 0.75] at (0.4, -0.8) {$\bullet$};
\end{tikzpicture}
\, - \,
\begin{tikzpicture}[baseline={([yshift=-.6ex]current bounding box.center)}, scale = 0.55, transform shape]
\resolveslice
\node[scale = 0.75] at (0.4, 0.8) {$\bullet$};
\node[scale = 0.75] at (0.4, -0.8) {$\bullet\,\bullet$};
\end{tikzpicture}
\, - \,
\begin{tikzpicture}[baseline={([yshift=-.6ex]current bounding box.center)}, scale = 0.55, transform shape]
\resolveslice
\node[scale = 0.75] at (0.4, -0.8) {$\bullet$};
\node[scale = 0.75] at (0.4, 0) {$\bullet\,\bullet$};
\end{tikzpicture}
\, + \,
\begin{tikzpicture}[baseline={([yshift=-.6ex]current bounding box.center)}, scale = 0.55, transform shape]
\resolveslice
\node[scale = 0.75] at (0.4, 0) {$\bullet$};
\node[scale = 0.75] at (0.4, -0.8) {$\bullet\,\bullet$};
\end{tikzpicture}
\end{align*}
and this is the statement we need.
\end{proof}

\section{Computations}
The author has developed a computer programme, \verb+ knotjob+, which can perform computations of the $s$-invariants and various spectral sequences, and which is available from his website. Memory consumption and speed are noticeably worse when compared to ordinary Khovanov homology. While the $(5,6)$-torus knot works fairly well, the $(6,7)$-torus knot represents a challenge for a 32GB desktop PC.

\subsection{$S$-invariants}\label{subsec:sinv}

We calculated $s^p_{\mathfrak{sl}_3}(K)$ for $p = 0, 2, 3, 5, 7$ for all prime knots with up to $14$ crossings, and found $372$ knots for which $s^p_{\mathfrak{sl}_3}(K) \not= s^p(K)$ for at least one $p\in \{0,2,3,5,7\}$, and where $s^p(K)$ denotes the $s$-invariant from Khovanov homology in characteristic $p$. In fact, for all of these knots we have $s^p_{\mathfrak{sl}_3}(K) = s^{p'}_{\mathfrak{sl}_3}(K)$ for $p,p'\in \{0,2,5,7\}$, and $s^0_{\mathfrak{sl}_3}(K) \not= s^3_{\mathfrak{sl}_3}(K)$. Furthermore, for all of these knots we have $s^3_{\mathfrak{sl}_3}(K)\in 2\Z$, and $s^p_{\mathfrak{sl}_3}(K)\notin 2\Z$ for all $p\in \{0,2,5,7\}$.

There are two knots with up to $14$ crossings for which $s^0_{\mathfrak{sl}_3}(K) \not= s^2_{\mathfrak{sl}_3}(K)$, but these are the same two knots for which $s^0(K)\not= s^2(K)$. One of these knots is the Whitehead double $D_+(T(2,3), 3)$. Another interesting Whitehead double in this vein is $D_+(T(3,4),8)$, the first observed example \cite{LewarkZib} where $s^p(K) = 0$ for all $p\not= 3$, and $2 = s^3(K)$. We get a similar behaviour for the $\mathfrak{sl}_3$-s-invariant.
\[
s^p_{\mathfrak{sl}_3}(D_+(T(3,4),8)) = \left\{
\begin{array}{cc}
0 & p\in \{0, 2, 5\}, \\
2 & p = 3.
\end{array}
\right.
\]

For prime knots with $15$ and $16$ crossings we only computed $s^p_{\mathfrak{sl}_3}(K)$ for $p\in \{2,3,7\}$. We found $17,901$ knots ($2,490$ with $15$ crossings, $15,411$ with $16$ crossings) where at least one of $s^p_{\mathfrak{sl}_3}(K)$ differs from $s^0(K)$. Table \ref{tb:slice_improv} lists the number of knots with $s^p_{\mathfrak{sl}_3}(K) > s^p(K)$ for various values of $p$.

\begin{table}[ht]
\begin{tabular}{|c|c|c|c|c|}
\hline
\backslashbox{\!$c$\!}{\!$p$\!} & 2 & 3 & 7 & Total \\
 \hline
 12 & 1 & 1 & 1 & 1\\
 13 & 7 & 4 & 7 & 7 \\
 14 & 55 & 45 & 55 & 55 \\
 15 & 57 & 397 & 448 & 459 \\
 16 & 3327 & 2689 & 3306 & 3342 \\
 \hline 
 \end{tabular}
 \vspace{0.2cm}
 \caption{\label{tb:slice_improv}The number of knots with $c$ crossings where $s^p_{\mathfrak{sl}_3}(K) > s^p(K)$.} 
 \end{table}

For prime knots with $17$ crossings we restricted ourselves to those hyperbolic knots whose Alexander polynomial satisfies the Fox-Milnor condition, the signature and $s^0(K)$ equal $0$. Of the remaining $120,890$ knots, $239$ have a non-zero $s^p_{\mathfrak{sl}_3}(K)$ with $p=2,3$ or $7$. Of these, $238$ are non-zero for $p=2$ and $7$, with $224$ are non-zero for $p = 3$.

For $18$ crossing prime knots we restricted ourselves even further, additionally requiring that the complete invariant of \cite{dunfield2024} has to vanish as well. Among the first $15$ million non-alternating hyperbolic knots (these knots, taken from \cite{MR4117738}, are ordered by increasing hyperbolic volume) we found $283,158$ such knots, and for $534$ of them  we computed $s^p_{\mathfrak{sl}_3}(K) \not=0$ for at least one $p\in\{ 2, 3, 5\}$. Again, for $p=3$ we have the most knots with vanishing $s$-invariant, only $466$ knots have a non-zero invariant. For $p=2$, resp.\ $p=5$, the invariant is non-zero for $531$, resp.\ $532$, knots. Finally, we calculated $s^0_{\mathfrak{sl}_3}$ for these $534$ knots and found that it agrees with $s^5_{\mathfrak{sl}_3}$ there.

Among these knots, the knot $K=18^{nh}_{1478167}$ stands out in that it has the property that $s_{\mathfrak{sl}_3}^2(K) = 1$, while for all other $p\leq 19$ we have $s^p_{\mathfrak{sl}_3}(K) = 0$, and all the slice obstructions from the spectral sequences below also vanish. Together with $18^{nh}_{272257}$, where $s^p_{\mathfrak{sl}_3}$ vanishes for $p=2$, but not for $p=0, 3$, and $18^{nh}_{5561369}$, where $s^p_{\mathfrak{sl}_3}$ vanishes for $p = 0, 2$, but not for $p = 3$, we see that $s^p_{\mathfrak{sl}_3}$ are linearly independent homomorphisms for $p=0, 2, 3$, even when restricted to the kernel of all the $s$-invariants from Khovanov homology.

In almost all examples of prime knots $K$ with $s^p_{\mathfrak{sl}_3}(K)\not= s^p(K)$ and $p\not= 3$ we calculated, we have $s^p_{\mathfrak{sl}_3}(K)\notin 2\Z$. By Theorem \ref{thm:first_main} this cannot happen for $p=3$. It would be interesting to know whether $s^p_{\mathfrak{sl}_3}(K)$ being odd has other consequences for $K$. One observation we made is that for such $K$ the numbers $s'_{X^3-X}(K), s''_{X^3-X}(K)$ and $s'''_{X^3-X}(K)$ tend to not all agree. Indeed, this happens for all knots with up to $14$ crossings and $s^0_{\mathfrak{sl}_3}(K)$ odd.

\subsection{Spectral Sequences in characteristic $0$}

Let $f(X) = X^3 - X - w$ with $w\in \C$. Lewark--Lobb \cite{MR3458146} initiated the study of $\KR$-equivalence and showed that there are knots where the polynomial with $w=0$ is not $\KR$-equivalent to the one with $w=1$. By the results from Section \ref{sec:kr_equiv} we can assume that real-part and imaginary-part of $w$ are non-negative.

A necessary and sufficient criterion for two polynomials to be $\KR$-equivalent for a link is that the associated spectral sequences are identical \cite[Prp.3.9]{MR3458146}. We use the notation $_wE^{i,j}_k$ for the spectral sequence  to indicate the dependence on $w\in \C$, noting that $_wE^{i,j}_1={_wE}^{i,j}_2=\Hsl^{i,j}(L;\C)$, and $_wE^{i,j}_3 = {_0E}^{i,j}_3 = {_0E}^{i,j}_4$ for all $w\in \C$. We are therefore particularly interested in bi-gradings $(i,j)$ and $w\in \C$ with $_wE^{i,j}_3 \not={_wE}^{i,j}_4$.

The first such examples were observed in \cite[Tb.3]{MR3458146}, and include $10_{125}\# 10_{125}$ for $w=1$. Indeed, a good recipe for getting such knots is to look at $K\# K$, where $K$ is a knot with $s^0_{\mathfrak{sl}_3}(K)\notin 2\Z$. However, we do not know if this holds in general.

\verb+knotjob+ can calculate $_wE^{i,j}_k$ for various values of $w$, including $w=n, \sqrt{n}$ for small non-negative integers, $w=\sqrt{-1}$, and for transcendental $w$. However, we have not observed much variation when varying $w$, due to the following result and the fact that our calculations are limited to fairly small knots due to memory restrictions.

\begin{proposition}\label{prp:kr_crit}
Let $L$ be a link. Then $_wE^{i,j}_4 = {_{w'}E}^{i,j}_4$ for all $w,w'\in \C-\{0\}$. Additionally, assume that the spectral sequence corresponding to the polynomial $X^3-X^2$ collapses at the $E_3$-page, and the spectral sequence corresponding to the polynomial $X^3-X$ collapses at the $E_5$-page. Then all $w,w' \in \C-\{\pm \frac{2\sqrt{3}}{9}\}$ give rise to $\KR$-equivalent polynomials over $L$.
\end{proposition}

\begin{proof}
Consider the polynomial ring in two variables $R=\Q[h, t]$, which is graded by $|h|_q = 2$ and $|t|_q=0=|1|_q$. Consider the polynomial $f(X) = X^3-h^2X-h^3t\in R[X]$ and let $D$ be a diagram for the link $L$. Then $\Csl(D;\Fr_f)$ is a free graded $R$-cochain complex which has a basis of homogeneous elements. The units of $R$ agree with the units of the subring $\Q$, and we can do Gaussian elimination on $\Csl(D;\Fr_f)$ until we obtain a free graded $R$-cochain complex $C$ with a basis $c_1,\ldots,c_l$ of homogeneous elements such that
\[
C^{i,j}\otimes_R\Q \cong \Hsl^{i,j}(L;\Q),
\]
for the base change $\eta\colon R \to \Q$ sending $h, t$ to $0$. This means for a basis element $c$ that
\[
\partial c = \partial^4 c+\partial^6 c +\cdots,
\]
where $\partial^\ell c$ is a linear combination of basis elements $c_{i_1},\ldots, c_{i_\ell}$ with $|c_{i_j}|_q = |c|_q - \ell$. Furthermore, for $\ell = 4, 8$ we have
\begin{equation}\label{eq:boundary_4}
\partial^\ell c = h^{\ell/2} \left(\sum a_{i_\ell}c_{i_\ell}\right) \hspace{0.4cm}\mbox{with }a_{i_\ell}\in \Q,
\end{equation}
and
\begin{equation}\label{eq:boundary_6}
\partial^6 c = h^3 t \left(\sum a_{i_6}c_{i_6}\right) \hspace{0.4cm}\mbox{with }a_{i_6}\in \Q.
\end{equation}
Now consider the base change $\chi\colon R \to \Q[t]$ sending $h$ to $1$. The basis elements of $\tilde{C} = C\otimes_R \Q[t]$ still have a well-defined $q$-degree, although the complex is now only filtered. But we still have a spectral sequence $(E_k)$ over $\Q[t]$, with $E_1^{i,j} = E_2^{i,j} = \tilde{C}^{i,j}$. Moreover, $E_3$ is obtained from $\tilde{C}$ using $\tilde{\partial}^4$ and Gaussian elimination, and by (\ref{eq:boundary_4}) we get that $E^{i,j}_3$ is a free $\Q[t]$-module in every bi-degree. The boundary map $d_3\colon E^{i,j}_3\to E^{i+1, j-6}$ is induced by $\partial^6$ and summands coming from the zig-zags of Gaussian elimination. But these zig-zags involve either $\partial^6|\circ (\partial^4|)^{-1}\circ \partial^4|$ or $\partial^4|\circ (\partial^4|)^{-1}\circ \partial^6|$. Either way, because of (\ref{eq:boundary_6}) we get that $d_3$ is of the form
\[
d_3(c) = t\left(\sum b_{i_6} c_{i_6}\right) \hspace{0.4cm}\mbox{with }b_{i_6}\in \Q.
\]
Now consider the base change $\sigma\colon \Q[t] \to \C$ sending $t\to w\in \C-\{0\}$. Then $_wE_3^{i,j} = E^{i,j}_3\otimes_{\Q[t]}\C$ and $_wd_3 = d_3\otimes\id$. In particular, the $_wd_3$ for varying $w$ only differ by an invertible factor. Hence the $E_4$ page does not depend on $w\not= 0$.

If we additionally assume that the spectral sequence for the polynomial $X^3-X^2$ collapses at the $E_3$-page, we get that $d_3$ above needs to be the $0$-homomorphism. This is because $X^3-X^2$ is $\KR$-equivalent to $X^3-X\pm\frac{2\sqrt{3}}{9}$, and therefore covered by the argument above.

Recall how $E_3$ was obtained from $\tilde{C}$ using Gaussian elimination. Since we now have $d_3=0$, we get $E_4 = E_3$, and $d_4\colon E_4^{i,j}\to E_4^{i+1, j-8}$ is induced by $\partial^8$ and zig-zags from the Gaussian elimination. These zig-zags can involve $\partial^8|\circ (\partial^4|)^{-1}\circ \partial^4|$ or $\partial^4|\circ (\partial^4|)^{-1}\circ \partial^8|$, but also $\partial^6|\circ (\partial^4|)^{-1}\circ \partial^6|$. By (\ref{eq:boundary_4}) and (\ref{eq:boundary_6}) we get that
\[
d_4(c) = \sum (a_{i_8}+t^2b_{i_8}) c_{i_8}\hspace{0.4cm}\mbox{with }a_{i_8},b_{i_8}\in \Q.
\]
Again, we can do the base change $\sigma\colon \Q[t]\to \C$ sending $t$ to $w$, this time even allowing $w=0$. As before, for $w=\pm\frac{2\sqrt{3}}{9}$ we need to get $_wd_4 = 0$. This means that $\pm w$ is a root for all the quadratic polynomials $a_{i_8}+b_{i_8}t^2$ appearing in $d_4$. 

If the spectral sequence for $X^2-X$ collapses at the $E_3$-page, these quadratic polynomial also need $w=0$ as a root, making all of $d_4=0$. Since the total rank of the $E_3$ page is minimal, all other spectral sequences $_wE_3$ also collapse. 

If the spectral sequence for $X^2-X$ collapses at the $E_5$-page, but not sooner, we need $_0d_4\not=0$, and so at least one of these polynomials $a_{i_8}+b_{i_8}t^2$ is non-zero. But then no other $w\not= \pm \frac{2\sqrt{3}}{9}$ can be a root of any of these polynomials, and so $_wd_4$ is also just $_0d_4$ up to an invertible factor. It follows that $_wE^{i,j}_5 = {_0E}^{i,j}_5$ for all $w\not=\pm \frac{2\sqrt{3}}{9}$. Since $_0E_5$ has minimal rank, so do all the other $_wE_5$.
\end{proof}

Computations show that the spectral sequence for $X^3-X^2$ collapses at the $E_3$-page, and for $X^3-X$ at the $E_5$-page for all prime knots with up to $14$ crossings. 

\begin{corollary}
Let $K$ be a prime knot with less than $15$ crossings. Then $X^3-X-w$ is $\KR$-equivalent to $X^3-X$ for all $w\in \C-\{\pm\frac{2\sqrt{3}}{9}\}$.\hfill \qedsymbol
\end{corollary}

If the spectral sequence for $X^3-X^2$ satisfies $E_4\not= E_3$, we get that $X^3-X$ is not equivalent to $X^3-X-w$ for any $w$ by Proposition \ref{prp:kr_crit}. In particular, this gives a strategy for finding knots with different spectral sequences.

The smallest knot example with this property is the $(4,5)$-torus knot, and we also observed this for other torus knots with $4$ and $5$ strands, and also for the $(6,7)$-torus knot. For torus links we observed this for $(4,8)$ and $(5,5)$, but not for $(4,4)$ or $(4,6)$. 

Table \ref{tb:t45} shows the $E_3$ page for the $(4,5)$-torus knot for the spectral sequence corresponding to $X^3-X-w$. The two arrows indicate the boundary map $d_3\colon E_3\to E_3$ for $w\not=0$. If $w\not=\pm \frac{2\sqrt{3}}{9}$, the $d_4$ map will kill all the remaining generators in non-zero homological degree, as otherwise any survivors remain for the $E_\infty$-page. In particular, for $w,w'\in \C-\{0,\pm \frac{2\sqrt{3}}{9}\}$ we have $X^3-X-w$ and $X^3-X-w'$ are $\KR$-equivalent over the $(4,5)$-torus knot.

\begin{table}[ht]
\begin{tabular}{|c||c|c|c|c|c|c|c|c|c|c|c|c|c|}
\hline
\backslashbox{\!$q$\!}{\!$h$\!} & $-12$ & $-11$ & $-10$ & $-9$ & $-8$ & $-7$ & $-6$ & $-5$ & $-4$ & $-3$ & $-2$ & $-1$ & $0$ \\
\hline
\hline
$52$  & $1$ &   &   &   &   &   &   &   &   &   &   &   &   \\
\hline
$50$  &   &   &   &   &   &   &   &   &   &   &   &   &   \\
\hline
$48$  &   &   &   &   &   &   &   &   &   &   &   &   &   \\
\hline
$46$  &   & $1$ &   &   &   &   &   &   &   &   &   &   &   \\
\hline
$44$  &   & $1$ &   &   &   &   &   &   &   &   &   &   &   \\
\hline
$42$  &   &   &   & $1$ &   &   &   &   &   &   &   &   &   \\
\hline
$40$  &   &   &   &   &   & $1$ &   &   &   &   &   &   &   \\
\hline
$38$  &   &   & $1$ &   &   & $1$ &   & $1$ &   &   &   &   &   \\
\hline
$36$  &   &   &   &   &   &   &   & $1$ &   &   &   &   &   \\
\hline
$34$  &   &   &   &   & $1$ &   &   &   &   & $1$ &   &   &   \\
\hline
$32$  &   &   &   &   &   &   & $1$ &   &   &   &   &   &   \\
\hline
$30$  &   &   &   &   &   &   & $1$ &   & $1$ &   &   &   &   \\
\hline
$28$  &   &   &   &   &   &   &   &   & $1$ &   &   &   &   \\
\hline
$26$  &   &   &   &   &   &   &   &   &   &   & $1$ &   & $1$ \\
\hline
$24$  &   &   &   &   &   &   &   &   &   &   &   &   & $1$ \\
\hline
$22$  &   &   &   &   &   &   &   &   &   &   &   &   & $1$ \\
\hline
\end{tabular}
\\[0.2cm]
\begin{tikzpicture}[overlay]
\draw[->] (-4.3,7) -- (-3.6, 5.7);
\draw[->] (-3.3, 5.25) -- (-2.65, 4);
\end{tikzpicture}
\caption{\label{tb:t45}The $E_3$-page of the $(4,5)$-torus knot for the polynomial $X^3-X-w$. The boundary $d_3$ for $w\not=0$ is indicated by arrows. Notice that the spectral sequence has to collapse at the $E_5$-page, so that all polynomials with $w\not=0, \pm\frac{2\sqrt{3}}{9}$ are $\KR$-equivalent.}
\end{table}

Most of the examples that we found where $X^3-X-w$ for $w\not=0$ is not $\KR$-equivalent to $X^3-X$ have the property that the spectral sequence with respect to $X^3-X^2$ does not collapse at the $E_3$-page, but at the $E_4$-page. The spectral sequence with respect to $X^3-X$ still collapses at the $E_5$-page. For such knots we do not expect different $\KR$-equivalence classes for different $w\in \C-\{0,\pm\frac{2\sqrt{3}}{9}\}$.

Interestingly, for the $(6,7)$-torus knot the spectral sequence with respect to $X^3-X^2$ collapses at the $E_5$-page rather than the $E_4$-page. The spectral sequence for $X^3-X$ also collapses at the $E_5$-page. 
%\begin{table}[ht]
%\begin{tabular}{|c||c|c|c|c|c|}
%\hline
%\backslashbox{\!$q$\!}{\!$h$\!} & $-13$ & $-12$ & $-11$ & $-10$ & $-9$ \\
%\hline
%\hline
%$88$  &   & $1$ &   &   &     \\
%\hline
%$86$  &   &   &   &   &     \\
%\hline
%$84$  & $1$ &   &   &   &     \\
%\hline
%$82$  & $3$ &   &   &   &   \\
%\hline
%$80$   &  &  & $3$ &   &    \\
%\hline
%$78$   &  &  & $1$ &   & $2$    \\
%\hline
%$76$  &  & $1$ &   &   & $1$   \\
%\hline
%$74$  &   & $2$ &   &   &    \\
%\hline
%$72$  &   &   &   & $2$ &    \\
%\hline
%$70$  &   &   &   & $1$ &    \\
%\hline
%\end{tabular}
%\\[0.2cm]
%\begin{tikzpicture}[overlay]
%\draw[->, dashed] (0.75,2.6) -- (1.4, 0.95);
%\end{tikzpicture}
%\caption{\label{tb:t67}The $E_4$-page of the $(6,7)$-torus knot for the polynomial $X^3-X-w$ with $w\not=0$ in homological degrees between $-13$ and $-9$. The dashed arrow indicates a hypothetical $d_4$ of rank $2$ which would result in a collapse at the $E_6$-page. In all other homological degrees the spectral sequence has to stabilize at the $E_5$-page.}
%\end{table}

\begin{question}
Do there exist knots or links such that $X^3-X-w$ and $X^3-X-w'$ are not $\KR$-equivalent for $w,w'\in \C-\{0, \pm\frac{2\sqrt{3}}{9}\}$?
\end{question} 

We expect the answer to this question to be yes, but we may require that the spectral sequence for $X^3-X$ collapses only after the $E_5$-page. There exist $4$-component links with $12$ crossings and $3$-component links with $13$ crossings where the spectral sequence collapses at the $E_7$-page, but since their spectral sequence for $X^3-X^2$ collapses at $E_3$, we do not expect them to be candidates. The $(6,6)$-torus link has the property that the spectral sequence for $X^3-X$ collapses at the $E_7$-page, and the spectral sequence for $X^3-X^2$ collapses at $E_4$.

\subsection{Spectral Sequences in characteristic $3$}

Let $f(X) = X^3 - X^2 - w$ with $w  \in \{\pm 1\} \subset \F_3$. Unlike in characteristic $0$, it is not clear whether $1$ and $-1$ give rise to the same results, although we have not found a knot where the spectral sequences are different for these two values. Nevertheless, the spectral sequence is very different from the case $X^3-X$. We usually have $E_2\not= E_1$, and we found examples where it collapses at $E_6$. The three numbers $s'_f(K), s''_f(K), s'''_f(K)$ can be different, although this happens much less frequently than in characteristic $0$.

As a slice obstruction, the numbers $s'_f(K), s''_f(K), s'''_f(K)$ in characteristic $3$ appear more efficient than $s^3_{\mathfrak{sl}_3}$. Recall from our computations in Section \ref{subsec:sinv} that the $s$-invariant for $p = 3$ was less efficient than other characteristics. In nearly all cases of $17$ and $18$ crossing knots with one $s^p_{\mathfrak{sl}_3}(K)\not=0$ and $s^3_{\mathfrak{sl}_3}(K)=0$ we get a non-zero entry from the polynomial $X^3-X^2-1$.

\bibliography{KnotHomology}
\bibliographystyle{amsalpha}

\end{document}